\documentclass[12pt,reqno]{amsart}

\newcommand\version{August 13, 2025}

\usepackage{amsmath,amsfonts,amsthm,amssymb,amsxtra}
\usepackage{bbm} 
\usepackage{hyperref}	

\usepackage{esint}

\usepackage{xfrac}
\usepackage{todonotes, relsize, pgfplots, tikz, float}
\usetikzlibrary{calc,decorations.markings}
\tikzset{
	> = stealth,
	every pin/.style = {pin edge = {}},
	flow/.style = {decoration = {markings, mark=at position #1 with {\arrow{>}}},
		postaction = {decorate}
	},
	flow/.default = 0.5,
	main/.style = {color=#1, line width=0.5pt, line cap=round, line join=round},
	main/.default = black,
	fontscale/.style={font=\relsize{#1}},
}

\newtheorem{theorem}{Theorem}[section]
\newtheorem{proposition}[theorem]{Proposition}
\newtheorem{lemma}[theorem]{Lemma}
\newtheorem{corollary}[theorem]{Corollary}

\theoremstyle{definition}

\theoremstyle{remark}

\newtheorem{remark}[theorem]{Remark}

\numberwithin{equation}{section}

\numberwithin{equation}{section}

\newcommand{\1}{\mathbbm{1}}

\renewcommand{\epsilon}{\varepsilon}

\newcommand{\N}{\mathbb{N}}

\renewcommand{\phi}{\varphi}
\newcommand{\R}{\mathbb{R}}

\newcommand{\Sph}{\mathbb{S}}

\usepackage{mathtools, todonotes, thmtools}\usepackage{tikz}
\usetikzlibrary{datavisualization}
\usetikzlibrary{datavisualization.formats.functions}
\usetikzlibrary{patterns}

\let\oldtocsection=\tocsection

\let\oldtocsubsection=\tocsubsection

\let\oldtocsubsubsection=\tocsubsubsection

\renewcommand{\tocsection}[2]{\hspace{0em}\oldtocsection{#1}{#2}}
\renewcommand{\tocsubsection}[2]{\hspace{1em}\oldtocsubsection{#1}{#2}}
\renewcommand{\tocsubsubsection}[2]{\hspace{2em}\oldtocsubsubsection{#1}{#2}}

\usepackage{todonotes}
\setuptodonotes{color=white, bordercolor=white, textcolor=blue, size=\normalsize}
\newif\ifTodo
\Todotrue

\begin{document}

\begin{titlepage}
    \huge \title[Sharp quantitative integral inequalities for harmonic extensions]{Sharp quantitative integral inequalities\\ for harmonic extensions}
		\vspace{7cm}
\end{titlepage}

\date{\version}
\thanks{\copyright\, 2025 by the authors. This paper may be reproduced, in its entirety, for non-commercial purposes.}

\author{Rupert L. Frank}
\address[Rupert L.~Frank]{
	Mathematisches Institut, Ludwig-Maximilians-Universit\"at M\"unchen, Theresienstr.~39, 80333 M\"unchen, Germany, and Munich Center for Quantum Science and Technology, Schel\-ling\-str.~4, 80799 M\"unchen, Germany, and Mathematics 253-37, Caltech, Pasa\-de\-na, CA 91125, USA}
\email{r.frank@lmu.de}

\author{Jonas W. Peteranderl}
\address[Jonas W.~ Peteranderl]{Mathematisches Institut, Ludwig-Maximilians-Universit\"at M\"unchen, The\-resienstr.~39, 80333 M\"unchen, Germany}
\email{peterand@math.lmu.de}

\author{Larry Read}
\address[Larry Read]{Mathematisches Institut, Ludwig-Maximilians-Universit\"at M\"unchen, Theresienstr.~39, 80333 M\"unchen, Germany}
\email{read@math.lmu.de}
	
\begin{abstract}
    We prove a quantitative version of a sharp integral inequality by Hang, Wang, and Yan for both the Poisson operator and its adjoint. Our result has the strongest possible norm and the optimal stability exponent. This stability exponent is not necessarily equal to 2, displaying the same phenomenon that Figalli and Zhang observed for the $p$-Sobolev inequality.
\end{abstract}

\maketitle
\setcounter{page}{1}


\section{Introduction and main results}

The objective of this paper is twofold. Our first goal is to turn certain sharp functional inequalities for the Poisson operator and its dual into a quantitative form by adding a term that involves the distance to the set of optimizers. Our second goal is to shed some light on a phenomenon discovered by Figalli and Zhang \cite{Figalli2022} in the context of the $p$-Sobolev inequality in quantitative form. Namely, they showed that the stability exponent (that is, the exponent with which the distance to the set of optimizers enters the inequality) changes depending on whether $p$ is smaller or larger than $2$. We demonstrate that the same change in behavior occurs in our setting of the Poisson operator, thus supporting the idea that the phenomenon discovered by Figalli and Zhang has some universal features. While there have been previous instances where the Figalli--Zhang stability exponent was found for $p>2$ \cite{guerra2025,Frank2024B}, our work and the concurrent work \cite{wang2025} seem to be the first ones where it is found in the (more intricate) regime of $p<2$. In contrast to all these previous instances, including \cite{wang2025}, our setting involves a non-local (and therefore not a differential) operator.

\subsection{The HWY-inequality and its stability}
Let $d\geq 3$ be the dimension. We study the harmonic extension operators from the Euclidean hyperplane $\R^{d-1}$ to the upper half-space $\R^d_+\coloneqq\R^{d-1}\times\R_+=\R^{d-1}\times(0,\infty)$.

\subsubsection*{Harmonic extension to the upper half-space}
The Poisson kernel of the upper half-space $\R^d_+$ is given by
\begin{equation*}
    P(x,\xi)\coloneqq \frac
    2 {|\mathbb S^{d-1}|} \frac{x_d}{\left(|x'-\xi|^2+x_d^2\right)^{\frac{d}{2}}}\,, \qquad (x,\xi)\in \R^d_+\times \R^{d-1}\,,
\end{equation*} where $|\mathbb S^{d-1}|$ denotes the measure of the $(d-1)$-dimensional unit sphere $\mathbb S^{d-1}$. The kernel $P$ acts on suitably integrable functions $f:\R^{d-1}\to \R$ via
\begin{equation*}
    (Pf)(x)\coloneqq \int_{\R^{d-1}}P(x,\xi)f(\xi) \, \mathrm d\xi\,,\qquad  x\in \R^{d}_+\,.
\end{equation*}  
The sharp Hang--Wang--Yan (HWY) inequality \cite{HANG2009, Hang2008} has the form
\begin{equation}\label{eq:HWY}\tag{HWY}
   C_d \|Pf\|^{\frac{2(d-1)}{d-2}}_{L^{\frac{2d}{d-2}}(\R^d_+)}\leq \|f\|^{\frac{2(d-1)}{d-2}}_{L^{\frac{2(d-1)}{d-2}}(\R^{d-1})}
\end{equation} with $C_d\coloneqq (d^{d-1}|\mathbb S^{d-1}|)^{1/d}
$. Moreover, equality holds if and only if $f=f_{a,b,\xi_0}$ for some $a\in\R$, $b\in \R_+$, and $\xi_0\in \R^{d-1}$, where
\begin{equation}\label{eq:optimizers}
    f_{a,b,\xi_0}(\xi)\coloneqq ab^{\frac{d-2}{2}}f_0(b(\xi-\xi_0))\qquad \text{with} \qquad f_0(\xi)\coloneqq \left(1+|\xi|^2\right)^{-\frac{d-2}{2}},\quad \xi\in \R^{d-1}\,.
\end{equation}

In this article we bring the HWY-inequality into quantitative form by adding a term that measures the distance to the set of optimizers \eqref{eq:optimizers}. More specifically, our first main result states the following.

\begin{theorem}\label{thm:HWYstability} Let $d\geq 3$. There is a $c_d>0$ such that for all $f\in L^{\frac{2(d-1)}{d-2}}(\R^{d-1})$ we have
    \begin{align*}  
       \|f\|^{\frac{2(d-1)}{d-2}}_{L^{\frac{2(d-1)}{d-2}}(\R^{d-1})}&-C_d\|Pf\|^\frac{2(d-1)}{d-2}_{L^{\frac{2d}{d-2}}(\R^d_+)}\\&\geq c_d \inf_{a,b,\xi_0}
\left(\int_{\R^{d-1}}|f-f_{a,b,\xi_0}|^{\frac{2(d-1)}{d-2}}\,\mathrm d\xi+\int_{\R^{d-1}}(f-f_{a,b,\xi_0})^2|f_{a,b,\xi_0}|^{\frac2{d-2}}\,\mathrm d\xi\right),
    \end{align*}
    where the infimum is taken over all
    $a\in\R$, $b\in \R_+$, and $\xi_0\in \R^{d-1}$.
\end{theorem}

This stability result enjoys a number of interesting optimality and invariance properties, which we will discuss in detail below. Already at this point we want to emphasize that the \emph{stability exponents} $\frac{2(d-1)}{d-2}$ for the $L^\frac{2(d-1)}{d-2}$-norm and $2$ for the weighted $L^2$-norm are \emph{optimal}.

\subsubsection*{Dual operator on the upper half-space}
The HWY-inequality states that $C_d^{-\frac{d-2}{2(d-1)}}$ is the norm of the operator $P:L^\frac{2(d-1)}{d-2}(\R^{d-1})\to L^\frac{2d}{d-2}(\R^d_+)$. Since the norm of an operator coincides with the norm of its adjoint, which in the case of $P$ is given by
\begin{equation*}
    (Tg)(\xi)\coloneqq \int_{\R^d_+} P(x,\xi) g(x) \,\mathrm  dx\,, \qquad \xi\in \R^{d-1} \,,
\end{equation*}
for suitably integrable functions $g:\R^d_+\to \R$, we arrive at the sharp dual HWY-inequality
\begin{equation}\tag{HWYd}\label{eq:HWYdual}
   C_d' \, \|Tg\|^{\frac{2d}{d+2}}_{L^{\frac{2(d-1)}{d}}(\R^{d-1})}\leq \|g\|^{\frac{2d}{d+2}}_{L^{\frac{2d}{d+2}}(\R_+^{d})}\, 
\end{equation}
with $C_d'\coloneqq(C_d)^\frac{d(d-2)}{(d+2)(d-1)}$. From the characterization of the optimizers in \eqref{eq:HWY}, we find that equality holds in \eqref{eq:HWYdual} if and only if $g=g_{a,b,\xi_0}$ for some $a\in \R,\, b\in \R_+,\, \xi_0\in \R^{d-1}$, where
\begin{equation*}
    g_{a,b,\xi_0}(x)\coloneqq ab^{\frac{d+2}{2}}g_0(b(x-\xi_0)) \qquad \text{with}\qquad g_0(x)\coloneqq |x+e_d|^{-d-2},\quad x\in \R^{d}_+\,,
\end{equation*}
and with the unit vector $e_d=(0,\ldots,1)\in\R^d$.

We bring the dual HWY-inequality into quantitative form as well, by including a term that measures the distance to the optimizers. Our second main result states the following. 

\begin{theorem}\label{thm:HWYstabilitydual} Let $d\geq 3$. There is a $c_d>0$ such that for all $g\in L^{\frac{2d}{d+2}}(\R^{d}_+)$ we have
\begin{equation*}
        \|g\|^\frac{2d}{d+2}_{L^{\frac{2d}{d+2}}(\R_+^{d})}-C_d' \,\|Tg\|^\frac{2d}{d+2}_{L^{\frac{2(d-1)}{d}}(\R^{d-1})}\geq
          c_d \inf_{a,b,\xi_0}\left(\int_{\R^d_+}|g-g_{a,b,\xi_0}|^{\frac{2d}{d+2}}
          \,\mathrm dx\right)^{\frac{d+2}d},
    \end{equation*}
    where the infimum is taken over all $a\in\R$, $b\in \R_+$, and $\xi_0\in \R^{d-1}$. 
\end{theorem}

As in Theorem \ref{thm:HWYstabilitydual}, the \emph{stability exponent} $2$ for the $L^\frac{2d}{d+2}$-norm is \emph{optimal}.

\subsubsection*{Discussion of the main results}
Let us make the following remarks.
\begin{enumerate}
    \item[(a)] Theorems \ref{thm:HWYstability} and \ref{thm:HWYstabilitydual} are stability results for the primal and dual HWY-inequalities, respectively, with right sides that vanish if and only if the left sides vanish. In the setting of the dual inequality in Theorem \ref{thm:HWYstabilitydual}, the right side vanishes quadratically, as one would naively expect. In contrast, in Theorem \ref{thm:HWYstability}, the order of vanishing depends on the norm that is used. Vanishing in the stronger norm, namely the $L^\frac{2(d-1)}{d-2}$-norm, comes with a weaker power, namely $\frac{2(d-1)}{d-2}$, while vanishing in the weaker norm, namely the weighted $L^2$-norm, comes with the stronger power $2$. In Section \ref{sec:4} we will prove that all these exponents are \emph{optimal} for the respective notion of convergence. 

    \item[(b)] As discussed in (a), the stability exponent (or vanishing order) is $\max\{2,r\}$, where $r=\frac{2(d-1)}{d-2}$ and $r=\frac{2d}{d+2}$ in the primal and dual setting, respectively. (We ignore for the moment the weighted $L^2$-norm in Theorem \ref{thm:HWYstability} and return to it in (c) below.) Note that the $L^r$-norm is the `strong' norm in the respective inequalities. The stability exponent $\max\{2,r\}$ was first found by Figalli and Zhang \cite{Figalli2022} in the setting of the $p$-Sobolev inequality. There $r=p$ and again the $L^r$-norm is the `strong' norm in their inequality. We find it remarkable that this Figalli--Zhang phenomenon persists in our non-local framework. So far, to our knowledge, other instances where this phenomenon was observed were restricted to the realm of (local) differential operators; see \cite{guerra2025,Frank2024B,wang2025}, for instance.

    \item[(c)] As mentioned in (a), in the case $r>2$, our remainder term involves two different norms, and each one comes with its own stability exponent. This was first noted in \cite{Frank2024B} in the context of the $p$-Sobolev inequality and a certain inequality in conformal geometry; see also \cite{ignat2025}, where the importance of the weighted $L^2$-norm in applications is highlighted. Again, the present paper hints towards a certain universality of such two-term remainders.

    \item[(d)] The HWY-inequality and its dual are invariant under M\"obius transformations; see Appendix \ref{sec:confinv}. The distances to the set of optimizers that we use share this invariance.

    \item[(e)] The sharp inequalities \eqref{eq:HWY} and \eqref{eq:HWYdual} are equivalent to each other by duality. We doubt that this is the case for their quantitative forms given in Theorems \ref{thm:HWYstability} and \ref{thm:HWYstabilitydual}. In this respect, we mention a duality theory for quantitative inequalities developed by Carlen \cite{Carlen2017}, but, as far as we can see, this abstract theory does not allow one to recover the optimal stability exponents. This purported non-equivalence of Theorems \ref{thm:HWYstability} and \ref{thm:HWYstabilitydual} is reflected in the fact that the proof of the latter is somewhat more involved than that of the former. Intuitively, this can be understood from the need in Theorem~\ref{thm:HWYstabilitydual} to expand $L^r$-norms with $r<2$, instead of $r\geq 2$, to second-order.
\end{enumerate}


\subsection{Passing to the ball}

The HWY-inequality is not only invariant under translations, dilations, and rotations, but under the full group of M\"obius transformations of $\R^{d-1}\cup\{\infty\}$. By duality, the corresponding invariance is true for its dual. This is discussed further in Appendix~\ref{sec:confinv}. Via stereographic projection, the (extended) boundary $\R^{d-1}\cup\{\infty\}$ is conformally equivalent to the sphere $\Sph^{d-1}$, and the half-space $\R^d_+$ to $B^d=\{x\in\R^d:\ |x|<1\}$, the unit ball. As a consequence, \eqref{eq:HWY} and \eqref{eq:HWYdual} have equivalent versions on $\Sph^{d-1}$ and $B^d$. We found it convenient to prove Theorems \ref{thm:HWYstability} and \ref{thm:HWYstabilitydual} in this equivalent setting.

\subsubsection*{Harmonic extension to the ball}
 The Poisson kernel of the unit ball $B^d$ is given by
\begin{equation}\label{eq:Poikersph}
    Q(y,\omega)\coloneqq \frac{1-|y|^2}{|y-\omega|^d} 
\,, \qquad (y,\omega)\in B^d\times \mathbb S^{d-1}\,.
\end{equation} 
The kernel $Q$ acts on suitably integrable functions $u:\mathbb S^{d-1}\to \R$ via
\begin{equation}\label{eq:Poifctball}
    (Qu)(y)\coloneqq \int_{\mathbb S^{d-1}}Q(y,\omega)u(\omega) \, \mathrm d\mu(\omega)\,,\qquad  y\in B^d\,.
\end{equation} 
Here $\mathrm d\mu$ is the uniform probability measure on $\mathbb S^{d-1}$. Note that the Poisson kernel \eqref{eq:Poikersph} would have an additional factor $|\mathbb S^{d-1}|^{-1}$ if $\mathrm d\mu$ was not normalized; compare \cite[p.~145]{Stein1990}.

If we endow the integrals on $B^d$ with a uniform probability measure, denoted by $\mathrm d\nu\coloneqq d r^{d-1}\mathrm dr\mathrm d\mu$, the HWY-inequality takes the form
\begin{equation}\label{eq:HWYsphere}\tag{HWY$\circ$}
   \|Qu\|_{L^{\frac{2d}{d-2}}(B^d)}\leq \|u\|_{L^{\frac{2(d-1)}{d-2}}(\mathbb S^{d-1})}
\end{equation}
with equality if and only if $u$ is constant up to symmetries of the inequality. These, given by the M\"obius transformations of $\Sph^{d-1}$, act on functions $u:\mathbb S^{d-1}\to \R$ via 
\begin{equation*}
    (u)_\Psi\coloneqq J_\Psi^{\frac{d-2}{2(d-1)}}u\circ\Psi\,,
\end{equation*}
where $\Psi\colon \Sph^{d-1} \to \Sph^{d-1}$ is a M\"obius transformation of $\Sph^{d-1}$ and $J_\Psi$ is the corresponding Jacobian. The conformal invariance of \eqref{eq:HWYsphere} means that, for any $u\in L^{\frac{2(d-1)}{d-2}}(\mathbb S^{d-1})$,
$$
\|(u)_\Psi \|_{L^{\frac{2(d-1)}{d-2}}(\mathbb S^{d-1})} =  \|u\|_{L^{\frac{2(d-1)}{d-2}}(\mathbb S^{d-1})}
\qquad\text{and}\qquad
\|Q(u)_\Psi\|_{L^{\frac{2d}{d-2}}(B^d)}=
\|Qu\|_{L^{\frac{2d}{d-2}}(B^d)} \,.
$$
We discuss this in Appendix \ref{sec:confinv}.

Now we can rephrase Theorem \ref{thm:HWYstability} on the sphere.

\begin{theorem}\label{thm:HWYstabilitysphere} Let $d\geq 3$. There is a $c_d>0$ such that for all $0\neq u\in L^{\frac{2(d-1)}{d-2}}(\mathbb S^{d-1})$ we have
    \begin{equation}\label{eq:stabQ}
    1 -\frac{\|Qu\|^{\frac{2(d-1)}{d-2}}_{L^{\frac{2d}{d-2}}(B^d)}}{ \|u\|^{\frac{2(d-1)}{d-2}}_{L^{\frac{2(d-1)}{d-2}}(\mathbb S^{d-1})}}\geq c_d \, \inf_{\lambda,\Psi}\left(\|1-\lambda (u)_{\Psi}\|^{\frac{2(d-1)}{d-2}}_{L^{\frac{2(d-1)}{d-2}}(\mathbb S^{d-1})}+\|1-\lambda (u)_{\Psi}\|^{2}_{L^{2}(\mathbb S^{d-1})}\right),
    \end{equation}
    where the infimum is taken over all $\lambda\in\R$ and over all M\"obius transformations $\Psi$ of $\Sph^{d-1}$.
\end{theorem}

\subsubsection*{Dual operator on the unit ball}
On the ball $B^d$, the dual HWY-inequality has the form
\begin{equation}\label{eq:HWYspheredual}\tag{HWYd$\circ$}
   \|Sv\|_{L^{\frac{2(d-1)}{d}}(\mathbb S^{d-1})}\leq \|v\|_{L^{\frac{2d}{d+2}}(B^d)}
\end{equation}
with 
\begin{equation*}
    (Sv)(\omega)\coloneqq \int_{B^d}Q(y,\omega)v(y)\,\mathrm d\nu(y)\,, \qquad \omega\in \Sph^{d-1} \,,
\end{equation*}
for suitably integrable functions $v:B^d\to \R$. Recall that $Q(y,\omega)$ is given by \eqref{eq:Poifctball} and $\mathrm d\nu$ is the uniform probability measure on $B^d$. Equality holds in \eqref{eq:HWYspheredual} if and only if $v$ is constant up to M\"obius transformations. In particular, the M\"obius transformations $\Phi\colon B^d \rightarrow B^d$ act on functions $v$ via 
\begin{equation*}
[v]_\Phi\coloneqq J_\Phi^{\frac{d+2}{2d}}v\circ\Phi\,,
\end{equation*}
where $\Phi:B^d\to B^d$ is a M\"obius transformation and $J_\Phi$ is the corresponding Jacobian.

We can then restate Theorem \ref{thm:HWYstabilitydual} in this setting on the ball.

\begin{theorem}\label{thm:HWYstabilitydualsphere} Let $d\geq 3$. There is a $c_d>0$ such that for all $0\neq v\in L^{\frac{2d}{d+2}}(B^d)$ we have
    \begin{equation*}
      1 -\frac{\|Sv\|^{\frac{2d}{d+2}}_{L^{\frac{2(d-1)}{d}}(\mathbb S^{d-1})}}{ \|v\|_{L^{\frac{2d}{d+2}}(B^d)}^{\frac{2d}{d+2}}}\geq c_d \, \inf_{\lambda,\Phi}\|1-\lambda\,[v]_\Phi\|^{2}_{L^{\frac{2d}{d+2}}(B^d)}\,,
    \end{equation*}
    where the infimum is taken over all $\lambda\in\R$ and over all M\"obius transformations $\Phi$ of $B^d$.
\end{theorem}


\subsection{Some background}

The inequality \eqref{eq:HWYsphere} is the higher-dimensional analogue of an inequality used in the proof of the isoperimetric inequality in the plane by Carleman \cite{Carleman1921}. Part of the motivation for Hang, Wang, and Yan \cite{HANG2009} was the problem in conformal geometry of finding a metric with vanishing scalar curvature that minimizes the isoperimetric ratio and is conformal to a given metric on a compact manifold with boundary.

\medskip

The HWY-inequality can also be regarded as an integrated version of Beckner's hypercontractivity estimates \cite{beckner_sobolev_1992}; though, the latter do not hold on the full domain of integration. Indeed, the operator $Q_\rho$ from $L^p(\mathbb S^{d-1})$ to $ L^q(\mathbb S^{d-1})$, $1\leq p\leq q\leq \infty$, with integral kernel $$Q_\rho(\tilde\omega, \omega)\coloneqq Q(\rho\tilde \omega,\omega)\,,\qquad \omega,\tilde \omega\in \mathbb S^{d-1}\,,$$ is only a contraction for $\rho^2\leq \frac{p-1}{q-1}$, which is strictly smaller than $1$ for $p=\frac{2(d-1)}{d-2}$ and $q=\frac{2d}{d-2}$. Hence, the HWY-inequality gives additional control over $Q_\rho$ on $\frac{p-1}{q-1}<\rho^2\leq 1$.

\subsubsection*{Stability of functional and geometric inequalities}
Our main motivation comes from the currently very active field of study concerned with quantitative stability properties of functional and geometric inequalities. While \emph{qualitative} stability refers to the fact that optimizing sequences for the corresponding inequalities are relatively compact up to the symmetries of the problem, by \emph{quantitative} stability we mean that the inequalities themselves are strengthened by the addition of a term that measures, in some sense, the distance of a given configuration to the `closest' optimal configuration.

An important impetus for this question came from the study of the isoperimetric inequality \cite{fusco_2008}, but these ideas have been further developed in the last two decades in many other settings, including geometric inequalities \cite{figalli_2010,cicalese_2012,figalli_2017,christ_2017,figalli_2018,frank_2021,vanhintum_2022,vanhintum_2024}, geometric eigenvalue inequalities \cite{brasco_2015,fusco_2017,allen_2023}, functional eigenvalue inequalities \cite{Carlen2014,mazari_2022}, inequalities in harmonic analysis \cite{christ_2014,christ_2021,gomez_2024,frank_2025}, as well as Sobolev inequalities, to which we turn next. We emphasize that these references constitute by no means a complete bibliography and are only intended to give the reader some selected pointers to the large literature.

\subsubsection*{Stability for Sobolev-type inequalities}
In the 80's, Brezis and Lieb \cite{BREZIS198573} raised the question whether the non-negative difference between both sides of the Sobolev inequality can be controlled from below by a notion of distance to the set of optimizers. An affirmative answer was given by Bianchi and Egnell \cite{Bianchi1991}, who proved that the difference is bounded from below by the square of the distance in terms of the $\dot W^{1,2}(\R^d)$-norm. Both the choice of the norm and the stability exponent 2 are best possible. The strategy developed by Bianchi and Egnell is very robust and has become the main tool in investigating stability of functional inequalities; see, for example, \cite{Chen2013} as well as the lecture notes \cite{Frank2023}. In passing we mention some recent developments that bypass the compactness step that is inherent in the Bianchi--Egnell method and lead to explicit constants; see \cite{Bonforte2023,Dolbeault2023,dolbeault_2024,carlen_2025,chen_2024,chen_2025,chen_2025b}.

Already Bianchi and Egnell asked for an extension of their result to the $p$-Sobolev inequality with $1<p<d$. In the absence of a Hilbert space structure, proving stability with the optimal norm and the optimal exponent turned out to be rather difficult. After several partial results (including \cite{Cianchi2009,Figalli2015,Neumayer2019}), this problem was finally settled by Figalli and Zhang \cite{Figalli2022}. As we have already mentioned, when using the $\dot W^{1,p}(\R^d)$-norm to measure the distance to the set of optimizers, they showed that the sharp stability exponent is $\max\{2,p\}$. The stability exponent $p>2$ for $p$-Sobolev-like inequalities was also found in the context of M\"obius transformations \cite{guerra2025} or of an inequality for the total $\sigma_2$-curvature on the sphere \cite{Frank2024B}; see also \cite{koenig2025}. The paper \cite{Frank2024B} introduced the idea of using \textit{two} different notions of distance to the set of minimizers with different stability exponents. The same conclusion appeared independently later in \cite{ignat2025}, where its usefulness for finite element methods is demonstrated.

In general, quantitative stability inequalities with stability exponents larger than two correspond to some sort of degenerate stability. While in the setting of the $p$-Sobolev inequality with $p>2$ (and in our setting) this degenerate stability is mainly due to the behavior of $L^p$-norms in a two-bubble regime, there is a different mechanism leading to degenerate stability, which is based on a zero mode of the Hessian that is not due to symmetries. This appeared in \cite{Engelstein2020} and is expected to lead to the sharp stability exponent 4, as was demonstrated in several examples in \cite{Frank2023a}. Further works where this mechanism was explored include \cite{Frank2024,  brigati2024, andrade2024, nobili2025}.

\subsection{Overview and notation}
It is well understood that the Bianchi--Egnell strategy is a robust way of proving quantitative stability, with optimal stability exponent, for functional inequalities in the presence of a Hilbert space structure. We consider our work here as a step towards understanding to which extent the Figalli--Zhang strategy serves a similar purpose in the absence of a Hilbert space structure. In particular, we show that the ingredients in their strategy are not tied to the local nature of the gradient operator that appears in their work but have analogues for the nonlocal Poisson operator as well. Verifying these individual ingredients requires rather different ad hoc arguments, most notably in the proof of the non-linear spectral gap inequality in Proposition \ref{prop:anospecV}. Specific to our setting is the conformal invariance, which is explained in Appendix \ref{sec:confinv} and which plays a crucial role in our proofs.

The remainder of this paper is divided into three sections. We establish stability of the HWY-inequality in Section \ref{sec:2} and for its dual in Section \ref{sec:3}. Finally, Section \ref{sec:4} deals with the sharpness of the stability exponents.
Henceforth, all proofs are carried out on the ball and then transferred to the half-space by stereographic projection.

Throughout, we write $\|\cdot\|_r\coloneqq \|\cdot\|_{L^{r}(\mathbb S^{d-1})}$ and $\|\cdot\|_r\coloneqq \|\cdot\|_{L^{r}(B^d)}$ for $1\leq r\leq\infty$. It should always be clear from the context whether the functions are defined on the sphere or on the ball. The same applies to orthogonality with respect to $L^2(\mathbb S^{d-1})$ and $L^2(B^d)$. Moreover, we set
$$p\coloneqq \frac{2(d-1)}{d-2}\,,\qquad q\coloneqq \frac{2d}{d-2}$$ with H\"older conjugates
$$p'\coloneqq \frac{2(d-1)}{d}\,,\qquad q'\coloneqq \frac{2d}{d+2}\,.$$
Lastly, for $\leq$ valid up to a multiplicative constant, possibly depending on the dimension, we will write $\lesssim$ and $\gtrsim$ for the reverse.


\section{Stability for the HWY-inequality}
\label{sec:2}

Our goal in this section is to provide a proof of Theorems \ref{thm:HWYstability} and \ref{thm:HWYstabilitysphere}.


\subsection{Strategy of the proof}

Our proof of stability for the HWY-inequality, like many other stability proofs since the work of Bianchi and Egnell \cite{Bianchi1991}, follows a two-part argument consisting of a global-to-local reduction and a local bound. In this subsection we describe what these steps mean in the present set-up and explain how they lead to the main result, Theorem~\ref{thm:HWYstabilitysphere}.

The first part, namely the global-to-local reduction, consists of a concentration compactness principle. For the HWY-inequality on the half-plane \eqref{eq:HWY} this was shown by Hang, Wang, and Yan \cite[Theorem 3.1]{Hang2008}. Translating this to the sphere via stereographic projection already establishes this step.

\begin{proposition}[Global-to-local reduction, \cite{Hang2008}]
  \label{prop:qual}  Let $(u_n)\subset L^{p}(\mathbb S^{d-1})$ be a sequence of functions with \begin{equation*}
        \|u_n\|_{p}\to 1\qquad \text{and}\qquad \|Qu_n\|_{q}\to 1
   \end{equation*} 
   as $n\to\infty$. Then \begin{equation*}
      \inf_{\Psi, \lambda\in\{\pm 1\}} \|\lambda(u_n)_\Psi-1\|_{p}\to 0 \qquad \text{as} \ n\to\infty\,.
   \end{equation*}
\end{proposition}

Most of the work in this section concerns the second part of the strategy, namely proving a local bound. We summarize the outcome as follows.

\begin{proposition}[Local bound]\label{prop:localstability}
    There exists a constant $c_d > 0$ such that, for any sequence $(u_n) \subset L^{p}(\Sph^{d-1})$ satisfying $\|u_n\|_{p}=1$ for all $n$ and $ \inf_\Psi\|(u_n)_\Psi-1\|_{p}\to 0$ as $n\rightarrow \infty$, we have
    \begin{equation*}
        \liminf_{n\rightarrow\infty}\frac{1-\|Qu_n\|^p_{q}}{\inf_\Psi \left(\|(u_n)_\Psi-1\|_{p}^{p}+\|(u_n)_\Psi-1\|_{2}^2\right)}\geq c_d\,. 
    \end{equation*}
\end{proposition}

Combining both propositions yields our main stability result, Theorem \ref{thm:HWYstabilitysphere}, via a standard contradiction argument. We include it for the sake of completeness.

\begin{proof}[Proof of Theorem \ref{thm:HWYstabilitysphere}]
	By contradiction, assume that there is a sequence $(u_n)\subset L^{p}(\mathbb S^{d-1})$ with
	\begin{equation}
		\label{eq:thmproofass}
		\frac{ 1-\|Qu_n\|^p_{q}/\| u_n \|^p_p}{\inf_{\lambda,\Psi} \left(  \|\lambda\,(u_n)_\Psi-1\|_{p}^{p} +\left\|\lambda\,(u_n)_\Psi-1 \right\|^{2}_{2} \right)} \to 0
	\end{equation} as $n\to\infty$. As the quotient is $0$-homogeneous, we may normalize the sequence $(u_n)$ by $\| u_n \|_p = 1$ for all $n$. Since
	$$ \inf_{\lambda,\Psi} \left(\| \lambda\,(u_n)_\Psi-1\|_{p}^{p} +\left\|\lambda\,(u_n)_\Psi-1\right\|^{2}_{2}\right)\leq \|1 \|_{p}^{p}+\| 1 \|_{2}^2\,,
	$$ we deduce from \eqref{eq:thmproofass} that $ \|Qu_n\|_{q}\to 1$ as $n\to\infty$. Hence, Proposition \ref{prop:qual} implies that $\inf_{\Psi,\lambda\in\{\pm 1\}} \| \lambda (u_n)_{\Psi} - 1 \|_{p} \to 0$ as $n\to\infty$. Passing to a subsequence and replacing $u_n$ by $-u_n$ if necessary, we may assume that $\inf_{\Psi} \| (u_n)_{\Psi} - 1 \|_{p} \to 0$ as $n\to\infty$. As
    \begin{equation*}
		\inf_{\lambda,\Psi} \left(\| \lambda\,(u_n)_\Psi-1\|_{p}^{p} +\left\|\lambda\,(u_n)_\Psi-1\right\|^{2}_{2}\right) \leq \inf_{\Psi} \left(  \| (u_n)_{\Psi} - 1 \|_{p}^{p}+\| (u_n)_{\Psi} - 1 \|_{2}^2 \right),
	\end{equation*}
    an application of Proposition \ref{prop:localstability} leads to a contradiction with \eqref{eq:thmproofass}.
\end{proof}

\begin{remark}\label{rem:HWYstabilitysphere}
    For later purposes, we record the following strengthening of Theorem \ref{thm:HWYstabilitysphere}: For $u\neq 0$ the lower bound \eqref{eq:stabQ} holds, where on the right side we set $\lambda= \pm\|u\|_p^{-1}$ and minimize over the choice of sign rather than over all $\lambda\in\R$. Indeed, for $\|u\|_p=1$ this was shown in the previous proof, and for general $u\neq 0$ this follows by homogeneity. 
\end{remark}

Theorem \ref{thm:HWYstability} now follows by stereographic projection. The details are as follows.

\begin{proof}[Proof of Theorem \ref{thm:HWYstability}]
    Let $\mathcal S: \R^{d-1}\cup\{\infty\}\to\Sph^{d-1}$ and $\Sigma:\R^d_+\to B^d$ denote the inverse stereographic projection and its conformal extension; see Appendix \ref{sec:confinv}. Let further $f\in L^p(\R^{d-1})$ with $\|f\|_{L^p(\R^{d-1})}=1$ and 
    $$
    u \coloneqq |\Sph^{d-1}|^\frac1p \, J_{\mathcal S^{-1}}^\frac1p \, f\circ \mathcal S^{-1} \,.
    $$
    We will use \eqref{eq:confinv} and recall the explicit value $C_d = (d^{d-1}|\Sph^{d-1}|)^{1/d}$. Applying Theorem \ref{thm:HWYstabilitysphere} (or rather Remark \ref{rem:HWYstabilitysphere}), we obtain
    $$
    \| f\|_{L^p(\R^{d-1})}^p - C_d \|Qf\|_{L^q(\R^d_+)}^p = \|u \|_p^p - \|Q u\|_{q}^p \geq c_d \, \inf_{\Psi,\lambda\in\{\pm 1\}} \left( \| \lambda(u)_\Psi - 1 \|_p^p + \|\lambda (u)_\Psi -1\|_2^2 \right)
    $$
    We set
    $$
    f_{\Psi^{-1}} \coloneqq |\Sph^{d-1}|^{-\frac1p} J_{\mathcal S}^\frac1p \, (1)_{\Psi^{-1}}\circ\mathcal S \,.
    $$ 
    This is the function on $\R^{d-1}$ that corresponds under the above mapping $f\mapsto u$ to the function $(1)_{\Psi^{-1}}$ on $\Sph^{d-1}$. Note that for $\lambda\in\{\pm 1\}$, by a change of variables,
    \begin{align*}
        \| \lambda (u)_\Psi - 1 \|_p^p = \int_{\Sph^{d-1}} |u - \lambda (1)_{\Psi^{-1}}|^p\,\mathrm d\mu
        = \int_{\R^{d-1}} |f- \lambda f_{\Psi^{-1}}|^p\,\mathrm d\xi
        \intertext{and}
        \| \lambda (u)_\Psi - 1 \|_2^2 = \int_{\Sph^{d-1}} (u - \lambda (1)_{\Psi^{-1}})^2 (1)_{\Psi^{-1}}^{2-p} \,\mathrm d\mu
        = \int_{\R^{d-1}} (f- \lambda f_{\Psi^{-1}})^p f_{\Psi^{-1}}^{2-p} \,\mathrm d\xi\,.
    \end{align*}
    The claimed bound follows from the fact that the functions $f_{\Psi^{-1}}$, with $\Psi$ ranging through M\"obius transformations of $\Sph^{d-1}$, coincide with the functions $f_{a_*,b,\xi_0}$, with $(b,\xi_0)$ ranging through $\R_+\times\R^{d-1}$. Here $a_*\coloneqq 2^{\frac{d-2}{2}}|\mathbb S^{d-1}|^{-\frac{1}{p}}$ is fixed. 
\end{proof}


\subsection{Preliminaries on the expansion}

For the proof of the local bound, Bianchi--Egnell \cite{Bianchi1991} and many subsequent works made use of the fact that the `strong' norm in their inequalities is induced by a Hilbert space structure. This is not the case in our setting, where the strong norm is the one in $L^p$. The breakthrough in the non-Hilbertian setting came in the work of Figalli and Zhang \cite{Figalli2022}. They promoted the use of `elementary inequalities', which, when specialized to our setting, read as follows; see \cite[Lemmas 2.1 and 2.4]{Figalli2022} and also \cite[Lemma 3.2]{Figalli2015}. 

\begin{lemma}[Elementary inequalities, \cite{Figalli2022}]\label{lem:elemineq}
Let $\kappa>0$.
\begin{enumerate}
    \item[(a)] There is a constant $c_\kappa>0$ such that
for any $a\in \R$ we have the lower bound
    $$|1+a|^{p}\geq 1+pa+p\frac{1-\kappa}{2}\left(a^2+(p-2)\zeta(a)(1-|1+a|)^2\right)+c_\kappa|a|^{p}\,,$$
    where
    $$\zeta(a)\coloneqq\begin{cases}
       \,\quad 1 &\quad\text{if}\ a\in [-2,0]^c\,,\\ |1+a|^{p-1} & \quad\text{if}\ a\in [-2,0]\,.
    \end{cases}
    $$
     \vspace{0.05cm}
    \item[(b)] There is a constant $C_\kappa>0$ such that
for any $a\in \R$ we have the upper bound
    $$|1+a|^{q}\leq 1+qa+\left(\frac{q(q-1)}{2}+\kappa\right)a^2+C_\kappa|a|^{q}\,.$$
\end{enumerate}
\end{lemma}

Besides mimicking an expansion around $a=0$ up to second order, the positive term $c_\kappa|a|^p$ is the key term in (a). From this term the $L^p$-distance can be eventually retrieved.


\subsection{Spectral gaps}\label{subsec:specgap}

The goal of this subsection is to prove a `non-linear spectral gap' inequality. We recall that the quantity $\zeta$ was introduced in Lemma \ref{lem:elemineq}. We further denote by $\mathcal H^l$ the space spanned by spherical harmonics of degree $0$ and $1$, or equivalently the space of restrictions of affine linear functions on $\R^{d}$ to $\mathbb S^{d-1}$.
 
\begin{proposition}[Non-linear spectral gap inequality for $Q$]\label{prop:anospec}
     Let $d\geq 3$ and $\lambda\in (0,2)$. There is a $\delta=\delta(d, \lambda)>0$ such that for any $\phi\in L^{p}(\mathbb S^{d-1})\cap (\mathcal H^l)^\perp$ with $\|\phi\|_{p}< \delta$ we have
     $$
     \|\phi\|_2^2+(p-2)\int_{\mathbb S^{d-1}} \zeta(\phi)(1-|1+\phi|)^2\, \mathrm d\mu \geq (p-1)\frac{d+2+\lambda}{d}\|Q\phi\|_{2}^2\,.
     $$
\end{proposition}
    
Our proof of this proposition uses a compactness property of $Q$ from \cite[Corollary~2.1]{HANG2009}.

\begin{lemma}[Compactness properties of $Q$, \cite{HANG2009}]\label{lem:cpctQ}
    Let $1\leq r<\infty$ and $1\leq s<\frac{dr}{d-1}$. The operator $$Q:L^{r}(\mathbb S^{d-1})\to L^{s}(B^d)$$ is compact. In particular, the operator $Q:L^2(\mathbb S^{d-1})\to L^2(B^d)$ is compact.
 \end{lemma}

Next, we prove a linear spectral gap inequality for the operator $Q$.

 \begin{lemma}[Spectral gap inequality for $Q$]\label{lem:spec}
 Let $d\geq 3$. For any $\phi\in L^{2}(\mathbb S^{d-1})\cap (\mathcal H^l)^\perp$, we have
 \begin{equation}\label{eq:Qbdd}\|\phi\|_2^2\geq \frac{d+4}{d}\|Q\phi\|_{2}^2\,.\end{equation}
 \end{lemma}

 \begin{proof}
    For each $\ell\in\N_0$ let $(Y_{\ell,m})$ be an orthonormal basis of spherical harmonics of degree $\ell$, normalized with respect to the uniform probability measure $\mu$. We shall use the fact that
    \begin{equation*}
        (QY_{\ell,m})(y)=|y|^\ell Y_{\ell,m}\left(\frac{y}{|y|}\right)\,,\qquad y\in B^d\,.
    \end{equation*}
    This follows directly from the definition of spherical harmonics as restrictions to the sphere of homogeneous harmonic polynomials on $\R^d$; see also \cite{Stein1990}.  
 
    To prove \eqref{eq:Qbdd}, we decompose $\phi$ into spherical harmonics, writing $\phi=\sum_{\ell,m}\phi_{\ell,m}Y_{\ell,m}$. Then 
\begin{align*}
    \|Q\phi\|_{2}^2&=\int_{B^d}\left(\int_{\Sph^{d-1}} Q(y,\omega)\sum_{\ell,m}\phi_{\ell,m}Y_{\ell,m}(\omega)\,\mathrm d\mu(\omega)\right)^2\,\mathrm d\nu(y)\\
    &=d\int_0^1|y|^{d-1}\int_{\Sph^{d-1}}\sum_{\ell,m,\ell',m'}\phi_{\ell,m}\phi_{\ell',m'}|y|^{\ell}Y_{\ell,m}\left(\frac{y}{|y|}\right)|y|^{\ell'}Y_{\ell',m'}\left(\frac{y}{|y|}\right)\,\mathrm d\mu\left(\frac y{|y|}\right)\mathrm d|y|\\
    &=d\sum_{\ell,m}\int_0^1|y|^{d-1+2\ell}\,\mathrm d|y|\phi_{\ell,m}^2=\sum_{\ell,m}\frac{d}{2\ell+d}\phi_{\ell,m}^2\leq \frac{d}{d+4}\sum_{\ell,m}\phi_{\ell,m}^2=\frac{d}{d+4}\|\phi\|_2^2\,.
\end{align*} The penultimate step follows from $\phi_{0,m}=\phi_{1,m}=0$.
 \end{proof}

We are now in position to prove the non-linear spectral gap inequality in Proposition \ref{prop:anospec}.

\begin{proof}[Proof of Proposition \ref{prop:anospec}]
     By contradiction, assume that there is a non-trivial sequence $(\phi_n)\subset L^{p}(\mathbb S^{d-1})\cap (\mathcal H^l)^\perp$ with $\phi_n\to 0$ in $L^{p}(\mathbb S^{d-1})$ as $n\to\infty$ and
\begin{equation}\label{eq:contra}
       1+(p-2)\int_{\mathbb S^{d-1}} \zeta(\phi_n)\left(\frac{1-|1+\phi_n|}{\|\phi_n\|_2}\right)^2\, \mathrm d\mu < (p-1)\frac{ d+2+\lambda}{d}\left\|Q\frac{\phi_n}{\|\phi_n\|_2}\right\|^2_{2}.
   \end{equation} 
Then, we deduce $\phi_n\to 0$ in $L^{2}(\mathbb S^{d-1})$ from the convergence in $L^{p}(\mathbb S^{d-1})$ by H\"older's inequality. Note that (up to a subsequence) $\phi_n\to 0$ pointwise almost everywhere, and hence $\zeta(\phi_n)\to 1$ almost everywhere. We are going to use further that (up to a subsequence) $\hat\phi_n\coloneqq \phi_n/\|\phi_n\|_2 \rightharpoonup \hat\phi$ weakly in $L^2(\mathbb S^{d-1})$ for some $\hat \phi\in L^2(\mathbb S^{d-1})$. If we write $\psi_n\coloneqq \hat\phi_n-\hat \phi$, we can bound
\begin{align}\notag
\liminf_{n\to\infty}\int_{\mathbb S^{d-1}}\zeta(\phi_n)&\left(\frac{1-|1+\phi_n|}{\|\phi_n\|_2}\right)^2\,\mathrm d\mu\geq \liminf_{n\to\infty}\int_{\{|\phi_n|\leq 1\}}\zeta(\phi_n)\hat \phi_n^2\,\mathrm d\mu\\&\geq \liminf_{n\to\infty}\int_{\{|\phi_n|\leq 1\}}\zeta(\phi_n)(2\psi_n\hat \phi+\hat \phi^2)\,\mathrm d\mu\geq \|\hat\phi\|_2^2\,.\label{eq:wconv}
\end{align}
Let us provide some details concerning the proof of the last inequality in \eqref{eq:wconv}. To the term involving $\hat\phi^2$, we applied Fatou's lemma to obtain the lower bound $\|\hat \phi\|_2^2$.
To dismiss the term involving $2\psi_n\hat\phi$, we exploited that $\psi_n\rightharpoonup0$ weakly and
$\zeta(\phi_n)\mathbbm 1_{\{|\phi_n|\leq 1\}}\hat \phi\to \hat\phi$ strongly in $L^2(\mathbb S^{d-1})$. The latter follows by dominated convergence as
$\zeta(\phi_n)\mathbbm 1_{\{|\phi_n|\leq 1\}}\leq 1$ pointwise almost everywhere.
    
  Since $Q:L^2(\mathbb S^{d-1})\to L^{2}(B^d)$ is compact by Lemma \ref{lem:cpctQ}, we have $Q \hat \phi_n\to Q\hat\phi$ in $L^2(B^d)$. Inserting this and the lower bound in \eqref{eq:wconv} into \eqref{eq:contra}, we obtain in the limit $n\to\infty$ that
  $$
  1+ (p-2) \|\hat\phi\|_2^2 \leq (p-1) \frac{d+2+\lambda}{d} \|Q\hat\phi\|^2_{2} \,.
  $$
  Note that this implies, in particular, that $Q\hat\phi\neq 0$. Moreover, $\|\hat\phi_n\|_2=1$ implies $\|\hat\phi\|_2\leq 1$. Inserting this into the previous inequality we arrive at
  $$
  \|\hat\phi\|_2^2\leq \frac{d+2+\lambda}{d}\|Q\hat\phi\|_{2}^2\,.
  $$ 
    The orthogonality condition $\hat\phi_n\bot \mathcal H^l$ implies $\hat\phi \bot \mathcal H^l$. Therefore, recalling also $\lambda<2$ and $Q\hat\phi\neq 0$, the previous inequality contradicts the spectral gap inequality in Lemma \ref{lem:spec}, which concludes the proof.
 \end{proof}


\subsection{Comparable distances and almost orthogonality}

In this subsection we show that functions that are sufficiently close to the set of optimizers can be M\"obius transformed in a way that preserves the closeness but ensures in addition some almost orthogonality conditions. Here, `almost' means up to terms of lower order compared to the $L^p(\mathbb S^{d-1})$- and the $L^2(\mathbb S^{d-1})$-distance. 

 \begin{proposition}[Approximate orthogonality]\label{prop:almorthog}
    Let $(u_n)\subset L^{p}(\Sph^{d-1})$ be a sequence of functions with \begin{equation*}
        \|u_n\|_{p}=1\qquad \text{and}\qquad \inf_\Psi\|(u_n)_\Psi-1\|_{p}\to 0
   \end{equation*} for  $n\to\infty$. Then there exists a sequence of M\"obius transformations $\Psi_n$ such that $$r_n\coloneqq (u_n)_{\Psi_n}-1$$ satisfies, for all $n$ large enough, 
   \begin{equation}
   \label{eq:normcomp}
       \|r_n\|_{2}=\inf_\Psi\|(u_n)_{\Psi}-1\|_{2}\,,\qquad \|r_n\|_{p}\lesssim \inf_\Psi\|(u_n)_{\Psi}-1\|_{p}\,,
   \end{equation} 
   and 
   \begin{equation}\label{eq:ortho}
       \bigg\vert\int_{\Sph^{d-1}} r_n\,\mathrm{d}\mu\bigg\vert+\sum_{i=1}^d\bigg\vert\int_{\Sph^{d-1}} \omega_i r_n\,\mathrm{d}\mu\bigg\vert\lesssim \|r_n\|_2^2+\|r_n\|_{p}^p\,, 
   \end{equation}
   where $\omega_i$, $i=1,\dots, d$, are the coordinate functions on the sphere.
\end{proposition}

\begin{proof}[Proof of Proposition \ref{prop:almorthog}]
    \emph{Step 1.} We recall the explicit family $\Psi_\eta$, $\eta\in B^d$, of M\"obius transformations of $\Sph^{d-1}$ given in \eqref{eq:psieta}. For $u\in L^p(\Sph^{d-1})$ we consider the functional 
    \begin{equation}
        \label{eq:mapdist}
        B^d \ni\eta\mapsto \|(u)_{\Psi_\eta}-1\|_{2}^2 \,.
    \end{equation}
    We are interested in minimizing this functional, and we claim that if $\inf_\Psi\|(u)_\Psi -1\|_2<1$, then the functional attains its minimum in $B^d$.

    To prove this, we first note that the functional in \eqref{eq:mapdist} is continuous. Indeed, we can approximate $u$ by smooth functions in the $L^2(\mathbb S^{d-1})$-norm, uniformly in $\eta$ on a small ball that is compactly contained in $B^d$. By a standard $\epsilon/3$-argument, the continuity of \eqref{eq:mapdist} for given $u$ follows from the one for smooth functions.
    
    As a step towards proving that the infimum of \eqref{eq:mapdist} is attained, we note that, if $\eta\rightarrow \omega\in \Sph^{d-1}$, the Jacobian $J_{\Psi_\eta}$ vanishes on $\Sph^{d-1}\backslash\{\omega\}$. As a consequence of Fatou's lemma, we then obtain
    \begin{equation}
    \label{eq:infimummobiusbdry}
        \liminf_{|\eta|\rightarrow 1}\|(u)_{\Psi_\eta}-1\|_2^2\geq \|1\|_2^2 =1 \,. 
    \end{equation}

    Next, we note that
    \begin{equation}
        \label{eq:infimummobius}
            \inf_\Psi\|(u)_\Psi -1\|_2 = \inf_\eta \|(u)_{\Psi_\eta} - 1 \|_2 \,.
    \end{equation}
    To show this, we recall that for any M\"obius transformation $\Psi$ of $\Sph^{d-1}$, there are $\eta\in B^d$ and $A\in O(d)$ such that $\Psi =A\Psi_\eta$; see Corollary \ref{cor:mobiusclass}. It will be slightly more convenient for us to write this as $\Psi = \Psi_{A\eta}(A\,\cdot\,)$. Then a change of variables $\omega\mapsto A^{-1}\omega$ shows that for any $r>0$ ($r=2$ is relevant for us)
    $$
    \int_{\Sph^{d-1}} |(u)_\Psi - 1|^r\,\mathrm d\mu = \int_{\Sph^{d-1}} |(u)_{\Psi_{A\eta}} - 1|^r\,\mathrm d\mu \,. 
    $$
    This proves \eqref{eq:infimummobius}.

    We can now prove the claim made at the beginning of this step. As a consequence of \eqref{eq:infimummobiusbdry}, \eqref{eq:infimummobius}, and the assumption $\inf_\Psi\|(u)_\Psi -1\|_2<1$, we have $\liminf_{|\eta|\rightarrow 1}\|(u)_{\Psi_\eta}-1\|_2^2 > \inf_\Psi\|(u)_\Psi -1\|_2$. By continuity of \eqref{eq:mapdist} and local compactness of $B^d$, the infimum is attained, as claimed. 
   
    \medskip

    \emph{Step 2.} Consider now a sequence $(u_n)\subset L^p(\Sph^{d-1})$ with $\|u_n\|_p=1$ and $\inf_\Psi \|(u_n)_\Psi -1\|_p\to 0$. We consider $n$ so large that the latter infimum is smaller than 1. We apply Step 1 with $u$ replaced by $u_n$ and obtain a sequence $(\eta_n)\subset B^d$ such that $\eta\mapsto \|(u_n)_{\Psi_\eta}-1\|_2^2$ attains its minumum at $\eta_n$. We set $\Psi_n\coloneqq \Psi_{\eta_n}$.
    
    By construction, $r_n\coloneqq(u_n)_{\Psi_{\eta_n}}-1$ satisfies the equality in \eqref{eq:normcomp}. The inequality in \eqref{eq:normcomp} follows by the same argument as in \cite[Lemma 9]{Frank2024B}. Just replace the $W^{1,2}$-norm by the $L^2$-norm, $W^{1,4}$ by $L^{p}$, and \cite[Lemma 6]{Frank2024B} by the invariance of the $L^{p}$-norm under M\"obius transformations. Moreover, \cite[Lemma 7]{Frank2024B} can be used as well, but with dimension $\tilde d= d-1$ and $q=4\frac{\tilde d-1}{\tilde d-4}<\frac{4\tilde d}{\tilde d-4}$ instead of $p=2$.
    
    We turn to the proof of the almost orthogonality conditions in \eqref{eq:ortho}. To bound the first summand in \eqref{eq:ortho}, we use the elementary inequality
    \begin{equation*}
        ||1+\sigma|^p-1-p\sigma|\lesssim \sigma^2+|\sigma|^p
    \end{equation*}
    for $\sigma\in\R$. Integrating this expression with $\sigma=r_n(\omega)$ yields that 
    \begin{equation*}
        p\bigg|\int_{\Sph^{d-1}}r_n\,\mathrm{d}\mu\bigg|=\bigg|\int_{\Sph^{d-1}} \left( |(u_n)_{\Psi_{\eta_n}}|^p-1-pr_n\right) \mathrm{d}\mu\bigg|\lesssim\|r_n\|_2^2+\|r_n\|_p^p\,.
    \end{equation*}

    For the remaining terms in \eqref{eq:ortho}, we employ that the $L^2(\mathbb S^{d-1})$-distance is attained. Assume without loss of generality that the map \eqref{eq:mapdist} has a minimum at $\eta=0$. Then differentiating with respect to $\eta$ at $0$ we have 
    \begin{equation*}
       \nabla_\eta\Big\vert_{\eta=0}\int_{\mathbb S^{d-1}} r_n(u_n)_{\Psi_\eta}\,\mathrm d\mu=0\,.
    \end{equation*}
    First, assume $r_n$ to be differentiable. Then, from the calculations in \cite[Lemma 10]{Frank2024B}, 
    \begin{equation*}
        \partial_{\eta_i}(u_n)_{\Psi_{\eta}}\vert_{\eta=0}=(d-2)\omega_i(1+r_n)-2(e_i-\omega_i\omega)\cdot \nabla r_n\,,\qquad i=1,\dots,d\,.
    \end{equation*}
    Since $$2\int_{\mathbb S^{d-1}} 
    r_n(e_i-\omega_i\omega)\cdot \nabla r_n\,\mathrm d\mu(\omega)=\int_{\mathbb S^{d-1}} 
    (e_i-\omega_i\omega)\cdot \nabla (r_n^2)\,\mathrm d\mu(\omega)= (d-1)\int_{\mathbb S^{d-1}} 
    \omega_i r_n^2\,\mathrm d\mu(\omega)\,,$$
    we obtain
    \begin{equation*}
       \left|\int_{\mathbb S^{d-1}} r_n\omega_i\,\mathrm d\mu\right|=\frac{1}{d-2}\left|\int_{\mathbb S^{d-1}} \omega_ir_n^2\,\mathrm d\mu\right|\lesssim \|r_n\|^2_2\,.
    \end{equation*}
    By approximation, this holds for general $r_n$.
\end{proof}


\subsection{Local analysis}

In this subsection we combine the ingredients of the two previous subsections and finish the proof of the local bound, Proposition \ref{prop:localstability}.

Before doing this, we show that, for all quantities that are of interest for us, the almost orthogonality conditions in Proposition \ref{prop:almorthog} can be replaced by genuine orthogonality conditions up to an acceptable error. To this end, we denote the orthogonal projection onto $\mathcal H^l$ by $\Pi^l$. Recall also that $\zeta$ was defined in Lemma \ref{lem:elemineq}.

\begin{lemma}[Negligible spherical harmonics of lower degree]\label{eq:sphdecomp}
   For $r_n$ as in Proposition \ref{prop:almorthog}, we split $r_n= r_n^{l}+r_n^h$ with $r_n^l\coloneqq\Pi^l r_n$. Then, as $n\rightarrow \infty$,
    \begin{align}
        \|r_n\|_{2}^2&=\|r_n^h\|_2^2+o\left(\|r_n\|_2^2+\|r_n\|_{p}^p\right),\notag
        \\
        \|r_n\|_{p}^p&=\|r_n^h\|_{p}^p+o\left(\|r_n\|_2^2+\|r_n\|_{p}^p\right),\label{eq:split2}\\
        \|Qr_n\|_2^2&=\|Qr_n^h\|_2^2+o\left(\|r_n\|_2^2+\|r_n\|_{p}^p\right),\notag
    \end{align}
  and 
    \begin{align}\notag
\int_{\Sph^{d-1}}\zeta(r_n)(1-|1+r_n|)^2\,\mathrm{d}\mu=\int_{\Sph^{d-1}}\zeta(r_n^h)(1-|1+r_n^h|)^2\,\mathrm{d}\mu\\+o\left(\|r_n\|_2^2+\|r_n\|_{p}^p\right). \notag
    \end{align}
\end{lemma}

\begin{proof}
    We first note that for $f\in L^1(\Sph^{d-1})$ we have
    $$
    (\Pi^l f)(\omega) = \int_{\Sph^{d-1}} f(\omega')\,\mathrm d\mu(\omega') 
    + d \sum_{i=1}^d \int_{\Sph^{d-1}} \omega_i' f(\omega')\,\mathrm d\mu(\omega') \,\omega_i \,.
    $$

    To see the first statement in \eqref{eq:split2}, we use $\|r_n\|_2^2=\|r_n^l\|_2^2+\|r_n^h\|_2^2$ together with the almost-orthogonality conditions \eqref{eq:ortho} and obtain
  \begin{align*}
        \|r_n^l\|_2^2 &= \bigg(\int_{\Sph^{d-1}}r_n\,\mathrm{d}\mu \bigg)^2+ d \sum_{i=1}^d \bigg(\int_{\Sph^{d-1}}r_n\omega_i\,\mathrm{d}\mu \bigg)^2\\&\lesssim\left(\|r_n\|_2^2+\|r_n\|_p^p\right)^2.
    \end{align*} 
    Similarly, recalling the action of the harmonic extension operator $Q$ on spherical harmonics in the proof of Lemma \ref{lem:spec}, we obtain again from \eqref{eq:ortho} that
    \begin{align*}
        \|Q r_n^l\|_{2}^2 & = \bigg(\int_{\Sph^{d-1}}r_n\,\mathrm{d}\mu\bigg)^2+\frac{d^2}{d+2}\sum_{i=1}^d\bigg(\int_{\Sph^{d-1}}r_n\omega_i\,\mathrm{d}\mu\bigg)^2\\&\lesssim\left(\|r_n\|_2^2+\|r_n\|_p^p\right)^2. 
    \end{align*}

    To prove the remaining statements, we use the fact, which follows immediately from the above formulas for $\Pi^l f$, that
    $$\sup_{\omega\in\Sph^{d-1}} \left|\Pi^lf(\omega)\right| \leq \left| \int_{\Sph^{d-1}} f(\omega')\,\mathrm d\mathrm\mu(\omega')\right| + d \left( \sum_{i=1}^d \bigg(\int_{\Sph^{d-1}}f(\omega')\omega'_i\,\mathrm{d}\mu(\omega')\bigg)^2 \right)^{1/2}.
    $$
    As a consequence of this and the almost orthogonality conditions \eqref{eq:ortho}, we have 
    \begin{equation}\label{eq:normeqortho}
        \|r_n^l\|_{\tilde p}\lesssim\|r_n\|_2^2+\|r_n\|_p^p
        \qquad \text{for any}\ \tilde p>1\,.
    \end{equation}
    
    The second statement in \eqref{eq:split2} now follows from 
    \begin{align*}
        \int_{\Sph^{d-1}}||r_n|^{p}-|r_n^h|^p|\,\mathrm{d}\mu&\leq p\int_{\mathbb S^{d-1}}\max\{|r_n|,|r_n^h|\}^{p-1}|r^l_n|\,\mathrm{d}\mu
        \\&\lesssim \|r_n^l\|_p(\|r_n\|_p+\|r_n^l\|_p)^{p-1}
    \end{align*}
    together with \eqref{eq:normeqortho}.
    
    Finally, we deal with the term involving $\zeta$ by estimating
    \begin{align*}
        &\bigg|\int_{\Sph^{d-1}}\left(\zeta(r_n)(1-|1+r_n|)^2-\zeta(r_n^h)(1-|1+r_n^h|)^2\right)\,\mathrm{d}\mu\bigg|\\
        \leq&\bigg|\int_{\Sph^{d-1}}\zeta(r_n)\left((1-|1+r_n|)^2-(1-|1+r_n^h|)^2\right)\,\mathrm{d}\mu\bigg|\\
        &+\bigg|\int_{\Sph^{d-1}}\left(\zeta(r_n)-\zeta(r_n^h)\right)(1-|1+r_n^h|)^2\,\mathrm{d}\mu\bigg|\eqqcolon (1)+(2)\,.
    \end{align*}
    Note that $\zeta(v)=\mathbbm 1_{|1+v|> 1}+\mathbbm 1_{|1+v|\leq 1} |1+v|^{p-1}\leq 1$ and recall that $p\geq 2$. Thus, using the H\"older and the triangle inequality, the first term gives
    \begin{align*}
        (1)&\leq \int_{\Sph^{d-1}}|(1-|1+r_n|)^2-(1-|1+r_n^h|)^2|\,\mathrm{d}\mu
        \\
        &\leq \left( \|1-|1+r_n|\|_p + \|1-|1+r_n^h|\|_p \right) \||1+r_n|-|1+r_n^h|\|_{\frac{p}{p-1}}\leq (\|r_n\|_p+\|r_n^h\|_p)\|r_n^l\|_{\frac{p}{p-1}}\,.
    \end{align*} 
    Via the bound \eqref{eq:normeqortho} with $\tilde p = p$ and $\tilde p=p/(p-1)$, it follows that 
    \begin{equation*}
        (1)\lesssim  \left( \|r_n\|_p + \|r_n\|_2^2 + \|r_n\|_p^p \right) \left(\|r_n\|_2^2+\|r_n\|_p^p\right).
    \end{equation*}
    
    For the second term, we split $\Sph^{d-1}= \bigcup_{i=1}^4 E_n^{(i)}$ with
    \begin{align*}
        E_n^{(1)} & \coloneqq \{ |1+r_n|>1 \,,\ |1+r_n^h|>1 \} \,,
        &E_n^{(2)} &\coloneqq \{ |1+r_n|>1\geq |1+r_n^h| \} \,,\\
        E_n^{(3)} &\coloneqq \{ |1+r_n^h|>1\geq |1+r_n| \} \,,
        &E_n^{(4)} &\coloneqq \{ |1+r_n|\leq 1 \,,\ |1+r_n^h|\leq 1\} \,,
    \end{align*}
    and correspondingly  
   $$
    (2) \leq \sum_{i=1}^4 (2)_i 
    \qquad\text{with}\qquad
    (2)_i \coloneqq  \bigg|\int_{E_n^{(i)}} \left(\zeta(r_n)-\zeta(r_n^h)\right)(1-|1+r_n^h|)^2\mathrm{d}\mu\bigg| \,.
    $$  
    First, we note that, by the definition of $\zeta$, we have $(2)_1=0$. Second, in $E_n^{(2)}$ we have $|r_n^l|\geq |1+r_n| - |1+r_n^h|> 1- |1+r_n^h|$ and therefore 
   \begin{align*}
       (2)_2 &\leq \int_{E_n^{(2)}} (1-|1+r_n^h|^{p-1}) (1-|1+r_n^h|)^2 \,\mathrm{d}\mu\\
       &\leq \|(1-|1+r_n^h|)\1_{E_n^{(2)}}\|_2^2 \leq \|r_n^l\|^2_2 \lesssim \left(\|r_n\|^2_2+\|r_n\|_p^p\right)^2.
    \end{align*}
    Third, since $(1-|1+r_n|^{p-1})/(1-|1+r_n|)\leq p-1$ in $E_n^{(3)}$, we have
    $$
    (2)_3 \leq \int_{E_n^{(3)}} (1-|1+r_n|^{p-1}) (|1+r_n^h|-1)^2 \,\mathrm{d}\mu
    \leq (p-1) \int_{E_n^{(3)}} (1-|1+r_n|)(|1+r_n^h|-1)^2 \,\mathrm{d}\mu\,.
    $$
    In $E_n^{(3)}$ we have $|r_n^l|\geq |1+r_n^h|-|1+r_n|>1-|1+r_n|$ and therefore
   \begin{align*}
       (2)_3 &\leq (p-1)\int_{E_n^{(3)}} |r_n^l| (|1+r_n^h|-1)^2 \,\mathrm{d}\mu
       \leq (p-1)\int_{\Sph^{d-1}}|r^l_n||r_n^h|^2 \,\mathrm{d}\mu\\
       &\leq (p-1)\|r_n^l\|_\infty\|r_n^h\|_2^2 \lesssim \|r_n\|_2^2\left(\|r_n\|_2^2+\|r_n\|_p^p\right).
   \end{align*}
    Finally, in $E_n^{(4)}$ we use 
    \begin{equation*}
       ||1+r_n|^{p-1}-|1+r_n^h|^{p-1}|\leq (p-1)\max\{|1+r_n|,|1+r_n^h|\}^{p-2}|r_n^l|\leq (p-1)|r^l_n|
    \end{equation*}
    to find 
    \begin{align*}
       (2)_4 &\leq \int_{E_n^{(4)}}\left| |1+r_n|^{p-1}-|1+r_n^h|^{p-1} \right| (1-|1+r_n^h|)^2 \, \mathrm{d}\mu\\
       &\leq (p-1) \|r_n^l\|_\infty\|r_n^h\|_2^2\lesssim \|r_n\|_2^2\left(\|r_n\|_2^2+\|r_n\|_p^p\right).
    \end{align*}
    This concludes the proof.
\end{proof}

Finally, we present the proof of Proposition \ref{prop:localstability} by combining the elementary estimates from Lemma \ref{lem:elemineq} and the spectral gap-type inequality from Proposition \ref{prop:anospec}. 

\begin{proof}[Proof of Proposition \ref{prop:localstability}]

We start with a sequence $(u_n)\subset L^p(\Sph^{d-1})$ satisfying $\|u_n\|_p=1$ for all $n$ and $\inf_\Psi \|(u_n)_\Psi -1\|_p\to 0$ as $n\to\infty$. Proposition \ref{prop:almorthog} gives us a sequence $(\Psi_n)$ of M\"obius transformations such that $r_n\coloneqq (u_n)_{\Psi_n}-1$ satisfies the almost orthogonality conditions \eqref{eq:ortho}. We note that the second item in \eqref{eq:normcomp} gives that $\|r_n\|_p \to 0$. 

By conformal invariance, we have
$$
1 - \|Qu_n\|_{q}^p = \|1+ r_n\|_{p}^{p}-\|Q(1+ r_n)\|_{q}^{p} \,,
$$
and our task is to find a lower bound on the right side.

Applying Lemma \ref{lem:elemineq}, for any $\kappa>0$ there are $C_\kappa, c_\kappa>0$ such that for $r\in L^p(\mathbb S^{d-1})$ we have
\begin{align}
    \|Q(1+ r)\|_{q}^{q}\leq 1&+q\int_{B^d}Qr\,\mathrm d \nu+\left(\frac{q(q-1)}{2}+\kappa\right)\|Qr\|_{{2}}^2 +C_{\kappa}\|Qr\|^{q}_{q} \notag
\intertext{and}\notag
    \|1+r\|_{p}^p\geq1&+p\int_{\Sph^{d-1}} r\,\mathrm d \mu+\frac{1-\kappa}{2}p\int_{\mathbb S^{d-1}} \left( r^2 + (p-2) \zeta(r)(1-|1+r|)^2\right) \mathrm d\mu \\ &+c_\kappa\|r\|_{p}^p\,. \label{eq:FZbound}
\end{align}
Moreover, using the elementary bound $(1+a)^{p/q}\leq 1+a p/q$ for $a\geq -1$, it follows that
\begin{align}
    \label{eq:Nbound}
    \|Q(1+ r)\|_{q}^{p}\leq 1&+p\int_{B^d}Qr\,\mathrm d \nu+\left(\frac{p(q-1)}{2}+\frac pq \kappa\right)\|Qr\|_{{2}}^2 +\frac{p}{q}C_{\kappa}\|Qr\|^{q}_{q} \,.
\end{align}

Applying \eqref{eq:FZbound} and \eqref{eq:Nbound} to $r=r_n$ and noting that $\int_{B^d} Qr_n \,\mathrm d\nu= \int_{\Sph^{d-1}} r_n \,\mathrm d\mu$, we obtain
\begin{align*}
   \|1+ r_n\|_{p}^{p}-\|Q(1+ r_n)\|_{q}^{p} & \geq  \frac p2 (1-\kappa) \int_{\mathbb S^{d-1}} \left( r_n^2 + (p-2) \zeta(r_n)(1-|1+r_n|)^2\right) \mathrm d\mu \\
    & \quad -\left(\frac{p(q-1)}{2}+\frac pq \kappa\right)\|Qr_n\|_{2}^2  +c_\kappa\|r_n\|_{p}^p- \frac{d-1}{d}C_\kappa\|Qr_n\|^{q}_{q} \,.
\end{align*}

Turning the almost orthogonality conditions into genuine orthogonality conditions through Lemma~\ref{eq:sphdecomp} and applying the non-linear spectral gap inequality from Proposition \ref{prop:anospec}, we obtain for any fixed $\lambda<2$ that
\begin{align*}
    \|Qr_n\|_{2}^2 & = \|Qr_n^h\|_{2}^2 + o(\|r_n\|_2^2 + \|r_n\|_p^p) \\
    & \leq \frac1{p-1}\,\frac{d}{d+2+\lambda} \int_{\mathbb S^{d-1}} \left( (r_n^h)^2 + (p-2) \zeta(r_n^h)(1-|1+r_n^h|)^2\right) \mathrm d\mu\\
    & \quad + o(\|r_n\|_2^2 + \|r_n\|_p^p) \\
    & = \frac1{p-1}\,\frac{d}{d+2+\lambda} \int_{\mathbb S^{d-1}} \left( r_n^2 + (p-2) \zeta(r_n)(1-|1+r_n|)^2\right) \mathrm d\mu \\
    & \quad + o(\|r_n\|_2^2 + \|r_n\|_p^p) \,.
\end{align*}
In addition, we bound
$$
\| Qr_n\|_{q}^q \lesssim \|r_n\|_p^q = o(\|r_n\|_p^p)\,,
$$
and then we use the positive term $(c_\kappa/2) \|r_n\|_p^p$ to absorb all the errors $o(\|r_n\|_p^p)$. In this way, we arrive at the lower bound
\begin{align*}
    \|1+ r_n\|_{p}^{p}-\|Q(1+ r_n)\|_{q}^{p} & \geq \frac p2 \left( (1-\kappa) - \left(q-1+\frac 2q \kappa\right) \frac{1}{p-1} \frac{d}{d+2+\lambda} \right) \\
    & \quad \quad \times \int_{\mathbb S^{d-1}} \left( r_n^2 + (p-2) \zeta(r_n)(1-|1+r_n|)^2\right) \mathrm d\mu \\
    & \quad + (c_\kappa/2)\|r_n\|_p^p + o(\|r_n\|_2^2) \,.
\end{align*}
Noting that $\frac{q-1}{p-1} = \frac{d+2}{d}$, we see that for any $\lambda<2$, we can choose $\kappa=\kappa(d,\lambda)>0$ small enough such that $$
(1-\kappa) -  \left(q-1+\frac 2q \kappa\right)\frac{1}{p-1} \frac{d}{d+2+\lambda}
= 1-\frac{d+2}{d+2+\lambda}-\kappa \left(1+\frac{(d-2)^2}{d(d+2+\lambda)}\right)>0\,.
$$
(Half of) this positive quantity can be used to absorb the error $o(\|r_n\|_2^2)$. Hence, we finally arrive at the lower bound
\begin{equation*}
    \|1+ r_n\|_{p}^{p}-\|Q(1+ r_n)\|_{q}^{p}
    \gtrsim \|r_n\|_2^2+\|r_n\|_{p}^p\,.
\end{equation*} 
Since 
$$
\|r_n\|_2^2+\|r_n\|_{p}^p \geq \inf_\Psi \left( \|(u_n)_\Psi - 1\|_2^2 + \|(u_n)_\Psi - 1\|_p^p \right),
$$
we obtain the assertion of Proposition \ref{prop:localstability}.
\end{proof}


 \section{Stability for the dual HWY-inequality}
\label{sec:3}

Now we turn to the dual setting. Our goal in this section is to prove Theorems \ref{thm:HWYstabilitydual} and \ref{thm:HWYstabilitydualsphere}, and therefore a quantitative dual HWY-inequality with the sharp power 2 of the distance to the set of optimizers.

\subsection{Strategy of the proof}
As in the previous section, to prove stability for the dual HWY-inequality, we follow the two-step strategy that consists of a global-to-local reduction and a local bound. Those two steps are the content of the following two propositions.

\begin{proposition}[Global-to-local reduction]
  \label{prop:qualdual} Let $(v_n)\subset L^{q'}(B^d)$ be a sequence of functions with \begin{equation*}
        \|v_n\|_{q'}\to 1\qquad \text{and}\qquad \|Sv_n\|_{p'}\to 1
   \end{equation*} 
   as $n\to\infty$. Then \begin{equation*}
      \inf_{\Phi,\lambda\in\{\pm 1\}} \|\lambda [v_n]_\Phi-1\|_{q'}\to 0 \qquad \text{as} \ n\to\infty\,.
   \end{equation*} 
\end{proposition}

\begin{proposition}[Local bound]\label{prop:localstabilitydual}
    There exists a constant $c_d > 0$ such that, for any sequence $(v_n) \subset L^{q'}(B^d)$ satisfying $\|v_n\|_{q'}=1$ and $ \inf_{\Phi}\| [v_n]_\Phi-1\|_{q'}\to 0$ as $n\rightarrow \infty$, we have
   \begin{equation*}
\liminf_{n\rightarrow\infty}\frac{1-\|Sv_n\|^{q'}_{p'}}{\inf_{\Phi} \|[v_n]_\Phi-1\|_{q'}^{2}}\geq c_d\,. 
    \end{equation*}
\end{proposition}

Given these two propositions, Theorem \ref{thm:HWYstabilitydualsphere} follows in the same way as Theorem \ref{thm:HWYstabilitysphere}.

\begin{proof}[Proof of Theorem \ref{thm:HWYstabilitydualsphere}]
	By contradiction, assume that there is a sequence $(v_n)\subset L^{q'}(B^d)$ with
	\begin{equation}
		\label{eq:thmproofass2}
		\frac{ 1-\|Sv_n\|^{q'}_{p'}/\| v_n \|^{q'}_{q'}}{\inf_{\lambda,\Phi}   \| \lambda\,[v_n]_{\Phi}-1\|_{q'}^{2}} \to 0
	\end{equation} 
    as $n\to\infty$. As the quotient is $0$-homogeneous, we may normalize the sequence $(v_n)$ by $\| v_n \|_{q'} = 1$ for all $n$. Since
	$$ 
    \inf_{\lambda,\Phi} \| \lambda\,[v_n]_\Phi - 1\|_{q'}^{2}\leq \|1 \|_{q'}^{2}\,,
	$$  
	we deduce from \eqref{eq:thmproofass2} that $\|Sv_n\|_{p'}\to 1$ as $n\to\infty$. Hence, Proposition \ref{prop:qualdual} implies that $\inf_{\Phi,\lambda\in\{\pm1\}} \| \lambda [v_n]_\Phi - 1\|_{q'}\to 0$ as $n\to\infty$. Passing to a subsequence and replacing $v_n$ by $-v_n$ if necessary, we may assume that $\inf_{\Phi} \| [v_n]_\Phi - 1\|_{q'}\to 0$ as $n\to\infty$. As
    $$
    \inf_{\lambda,\Phi}   \| \lambda\,[v_n]_{\Phi}-1\|_{q'}^{2} \leq \inf_{\Phi}   \| [v_n]_{\Phi}-1\|_{q'}^{2} \,,
    $$
    an application of Proposition \ref{prop:localstabilitydual} leads to a contradiction with \eqref{eq:thmproofass2}.
\end{proof}

Theorem \ref{thm:HWYstabilitydual} can be deduced from Theorem \ref{thm:HWYstabilitydualsphere} in the same way as Theorem \ref{thm:HWYstability} was deduced from Theorem \ref{thm:HWYstabilitysphere}. We omit the details. The remainder of this section is devoted to the proof of Propositions \ref{prop:qualdual} and \ref{prop:localstabilitydual}.


\subsection{Global-to-local reduction via duality}

In this subsection we prove Proposition \ref{prop:qualdual}. Our strategy is to deduce the statement from Proposition \ref{prop:qual} via duality.

\begin{proof}[Proof of Proposition \ref{prop:qualdual}]
    Let $(v_n)\subset L^{q'}(B^d)$ satisfy, as $n\to\infty$, 
    \begin{equation*}
     \|v_n\|_{q'}\rightarrow 1 \qquad \text{and}\qquad\|Sv_n\|_{p'}\rightarrow 1 \,.
    \end{equation*}
    We claim that the sequence $(u_n)\subset L^p(\mathbb S^{d-1})$ defined by 
    \begin{equation*}
    u_n \coloneqq  \frac{|Sv_n|^{p’-2} Sv_n}{\|Sv_n\|_{{p’}}^{p’-2}}
    \end{equation*}
    satisfies, as $n\to\infty$, 
    \begin{equation}
        \label{eq:dualopt}
                \|u_n\|_p= \|Sv_n\|_{p'} \to 1 \qquad \text{and}\qquad  \|Qu_n\|_{q}\rightarrow 1 \,.
    \end{equation}
    The first relation is clear. To prove the second relation, we note that, on the one hand,
    \begin{equation*}
        \int_{\mathbb S^{d-1}} u_n(Sv_n)\,\mathrm d \mu=\|Sv_n\|_{p'}^2 \rightarrow 1\,,
    \end{equation*}
    while on the other hand, using H\"older's inequality and the HWY-inequality,
    \begin{equation*}
        \bigg\vert\int_{\mathbb S^{d-1}} u_n(Sv_n)\,\mathrm d \mu\bigg\vert=\bigg \vert \int_{B^d} (Qu_n)v_n\,\mathrm d \nu\bigg\vert\leq \|Qu_n\|_{q}\|v_n\|_{q'}\leq \|u_n\|_p \|v_n\|_{q'} \to 1\,.
    \end{equation*}
    Thus, we find $\|Qu_n\|_{q}\rightarrow 1$, as claimed.

    It follows from \eqref{eq:dualopt} via Proposition \ref{prop:qual} that 
    \begin{equation*}
        \inf_{\Psi,\lambda\in\{\pm1\}} \|\lambda (u_n)_{\Psi}-1\|_p\rightarrow 0\,.
    \end{equation*}
    Consequently, there is a sequence $(\Psi_n)$ of M\"obius transformations of $\Sph^{d-1}$ and a sequence $(\lambda_n)\subset\{\pm 1\}$ of signs such that $\tilde u_n\coloneqq  (u_n)_{\Psi_n}$ satisfies
    \begin{equation}
        \label{eq:dualoptproof0}
            \|\tilde u_n\|_p \to 1 \,,\qquad
        \| Q \tilde u_n\|_{q} \to 1
        \,,\qquad\text{and}\qquad
        \| \lambda_n \tilde u_n - 1 \|_p \to 0 \,.
    \end{equation}
    Indeed, the first two relations follow from \eqref{eq:dualopt} and conformal equivalence.

    Let $\Phi_n\coloneqq \tilde\Psi_n$ be the Poincar\'e extension of $\Psi_n$ as discussed in Appendix \ref{sec:confinv}. Then, by Lemma~\ref{lem:poincareball}, we have for all $u\in L^p(\Sph^{d-1})$,
    $$
    \int_{\Sph^{d-1}} \! J_{\Psi_n}^\frac1{p'} \, ((Sv_n)\circ\Psi_n) u\, \mathrm d\mu = \int_{B^d} v_n Q (u)_{\Psi_n^{-1}}\,\mathrm d\nu = \int_{B^d} v_n J_{\Phi_n^{-1}}^\frac1q \, (Qu)\circ\Phi_n^{-1} \,\mathrm d\nu = \int_{\Sph^{d-1}} \!(S[v_n]_{\Phi_n}) u\,\mathrm d\mu \,,
    $$
    and therefore
    \begin{equation}
        \label{eq:dualoptproof1a}
            J_{\Psi_n}^\frac1{p'} \, (Sv_n)\circ\Psi_n = S[v_n]_{\Phi_n} \,.
    \end{equation}

    We define
    $$
    \tilde v_n\coloneqq [v_n]_{\Phi_n}
    $$
    and note that, by conformal invariance and the corresponding properties of $v_n$,
    $$
    \|\tilde v_n\|_{q'}= \|v_n\|_{q'} \to 1
    \qquad\text{and}\qquad
    \| S\tilde v_n\|_{p'} = \| Sv_n\|_{p'} \to 1 \,.
    $$
    Moreover, by conformal invariance, the definition of $u_n$, and \eqref{eq:dualoptproof1a}, we have
    $$
    \tilde u_n = \left( \frac{|Sv_n|^{p'-2}Sv_n}{\|Sv_n\|_{p'}^{p'-2}} \right)_{\Psi_n} = \frac{|S\tilde v_n|^{p'-2}S\tilde v_n}{\|Sv_n\|_{p'}^{p'-2}} = \frac{|S\tilde v_n|^{p'-2}S\tilde v_n}{\|S\tilde v_n\|_{p'}^{p'-2}} \,.
    $$
    Using $\|\tilde u_n\|_p = \|u_n\|_p = \|Sv_n\|_{p'} = \|S\tilde v_n\|_{p'}$, this relation can be inverted to yield
    $$
    S\tilde v_n = \frac{|\tilde u_n|^{p-2} \tilde u_n}{\|\tilde u_n\|_p^{p-2}} \,.
    $$
    
    We now claim that
    \begin{equation}
        \label{eq:dualoptproof2}
        \lambda_n S\tilde v_n \to 1
        \qquad\text{in}\ L^{p'}(\Sph^{d-1}) \,.
    \end{equation}
    To prove this, we first note that, since $\|S\tilde v_n\|_{p'}\to 1$, we have, along a subsequence $\lambda_n S\tilde v_n\rightharpoonup z$ weakly in $L^{p'}(\Sph^{d-1})$ for some $z\in L^{p'}(\Sph^{d-1})$. The first and third relation in \eqref{eq:dualoptproof0} imply that
    $$
    \int_{\Sph^{d-1}} z\,\mathrm d\mu \leftarrow  \int_{\Sph^{d-1}} (\lambda_n\tilde u_n) (\lambda_nS\tilde v_n) \,\mathrm d\mu = \| \lambda_n\tilde u_n\|_p^2 = \| \tilde u_n\|_p^2 \to 1 \,,
    $$
    so $\int_{\Sph^{d-1}} z\,\mathrm d\mu =1$. This, together with
    $$
    \left| \int_{\Sph^{d-1}} z\,\mathrm d\mu \right| \leq \|1\|_p \|z\|_{p'} \leq 1
    $$
    and the characterization of cases of equality in H\"older's inequality implies that $z=1$. By uniqueness of the limit, we deduce that the whole sequence $(\lambda_n S\tilde v_n)$, and not only a subsequence, converges weakly to $1$. Since $\|S\tilde v_n\|_{p'} \to 1 = \|1\|_{p'}$, it follows from \cite[Proposition 3.32]{brezisbook} that the convergence is, in fact, strong, as claimed in \eqref{eq:dualoptproof2}.

    It is now easy to finish the proof of the proposition. Since $\|\tilde v_n\|_{q'}\to 1$, we have, along a subsequence, $\lambda_n \tilde v_n\rightharpoonup w$ weakly in $L^{q'}(B^d)$ for some $w\in L^{q'}(B^d)$. By \eqref{eq:HWYspheredual}, it follows that $\lambda_n S\tilde v_n \rightharpoonup Sw$ weakly in $L^{p'}(\Sph^{d-1})$ and therefore, by \eqref{eq:dualoptproof2}, that $Sw = 1$. In particular, $w\not\equiv 0$. Meanwhile, $\|\tilde v_n\|_{q'}\to 1$ implies that $\|w\|_{q'}\leq 1$. Thus,
    $$
    \frac{\|Sw\|_{p'}}{\|w\|_{q'}} = \frac{1}{\|w\|_{q'}} \geq 1 \,.
    $$
    Since the left side is $\leq 1$ by the dual HWY-inequality, we deduce that $\|w\|_{q'}=1$ and that $w$ is an optimizer for the dual HWY-inequality. By the explicit characterization of its optimizers (which, by duality, follows from the explicit characterization of the optimizers of the primal HWY-inequality) and the fact that $Sw=1$, we deduce that $w=1$. By uniqueness of the limit, we deduce that the whole sequence $(\lambda_n \tilde v_n)$, and not only a subsequence, converges weakly to $1$ in $L^{q'}(B^d)$. Since $\|\tilde v_n\|_{L^{q'}(B^d)} \to 1 = \|1\|_{L^{q'}(B^d)}$, it follows again from \cite[Proposition 3.32]{brezisbook} that the convergence is, in fact, strong. Thus, we have shown that
    $$
    \inf_{\Phi,\lambda\in\{\pm 1\}} \|\lambda [v_n]_{\Phi} - 1 \|_{L^{q'}(B^d)} \leq \|\lambda_n \tilde v_n - 1 \|_{L^{q'}(B^d)} \to 0 \,,
    $$
    as claimed.
\end{proof}

We emphasize that the preceding argument is of abstract nature and applicable to general operators $Q:X\to Y$ and their duals $S=Q':Y'\to X'$ between Banach spaces. The crucial assumption is uniform convexity of the spaces and their duals. In this setting $u_n\in X''=X$ is defined as the duality mapping applied to $Sv_n\in X'$. We recall that the duality map from $X'$ to $X''=X$ is well defined when $X$ uniformly convex (and consequently reflexive); see, for instance, \cite[Problem~13]{brezisbook}. Uniform convexity also ensures that if a sequence converges weakly with converging norms, then it converges strongly; see \cite[Proposition 3.32]{brezisbook}.


\subsection{Preliminaries on the expansion}

The remainder of this section is devoted to the proof of the local bound in Proposition \ref{prop:localstabilitydual}. In the dual setting, just like in the primal setting, we rely on the use of the elementary inequalities from \cite{Figalli2022}, which tailored to our needs can be stated as follows; see \cite[Lemmas 2.1 and 2.4]{Figalli2022}.

\begin{lemma}[Elementary inequalities, \cite{Figalli2022}]\label{lem:elemineq2}

Let $\kappa>0$.
\begin{enumerate}    
    \item[(a)] There is a constant $c_\kappa>0$ such that
for any $a\in \R$ we have the lower bound
    $$|1+a|^{q'}\geq 1+q'a+q'\frac{1-\kappa}{2}\left(a^2+(q'-2)\zeta(a)(1-|1+a|)^2\right)+c_\kappa\min\{|a|^{q'},a^2\}\,,$$
    where
    $$\zeta(a)\coloneqq\begin{cases}
       \frac{|1+a|}{(2-q')|1+a|+(q'-1)} & \quad\text{if} \ a\in [-2,0]^c\,,\\\qquad \quad 1 & \quad\text{if}\ a\in [-2,0]\,.
    \end{cases}$$
    
    \item[(b)] There is a constant $C_\kappa>0$ such that
for any $a\in \R$ we have the upper bound
    $$|1+a|^{p'}\leq 1+{p'}a+\left(\frac{p'(p'-1)}{2}+\kappa\right)\frac{(1+C_\kappa|a|)^{p'-2}}{1+a^2}a^2\,.$$
\end{enumerate}
\end{lemma}

Owing to the exponent $q'$ being smaller than $2$, the key distinction from the previous section lies in the structure of the second-order term. This complicates both the compactness argument and the derivation of spectral gap inequalities; see Subsection \ref{subsec:specgap2}. Remarkably, however, we are able to adapt the approach of \cite{Figalli2022} -- despite its explicit reliance on local operators -- to our setting involving the non-local dual Poisson operator $S$.


\subsection{Compactness of Orlicz type} 

To prove a non-linear spectral gap estimate that is compatible with the expansions from Lemma \ref{lem:elemineq2} we require a rather intricate compactness result.

\begin{lemma}[A compactness result for non-linear $L^2$-spaces]
    \label{lem:cpctresult} Let $(\phi_n)$ be a sequence in $L^{q'}(B^{d})$ and let $(\varepsilon_n)\subset\R_+$ with
    $\varepsilon_n\to 0$. If \begin{equation}\label{eq:bddassump}
        \int_{B^{d}}\frac{\phi_n^2}{(1+\varepsilon_n|\phi_n|)^{2-{q'}}}\,\mathrm d\nu\leq 1\,,
    \end{equation}
    then, along a subsequence, $\phi_n$ converges weakly in $L^{q'}(B^d)$ to a function $\phi\in L^{{q'}}(B^{d})$ with $S\phi\in L^2(\mathbb S^{d-1})$, and for any constant $C>0$ it holds that \begin{equation*}
        \lim_{n\to\infty}\int_{\mathbb S^{d-1}}\frac{(1+C\varepsilon_n|S\phi_n|)^{p'}}{1+\varepsilon^2_n(S\phi_n)^2}(S\phi_n)^2\,\mathrm d \mu=\int_{\mathbb S^{d-1}}(S\phi)^2\,\mathrm d \mu\,.
    \end{equation*}
\end{lemma}

This is a higher integrability result in the sense that the trivial inclusion $S\phi\in L^{p'}(\Sph^{d-1})$, which follows from $\phi\in L^{q'}(B^d)$, is improved to $S\phi\in L^2(\Sph^{d-1})$. (Recall that $2>q'$.)

Lemma \ref{lem:cpctresult} is the analogue of \cite[Lemma 3.4]{Figalli2022}. The use there of a Hardy--Poincar\'e inequality is replaced here by an application of Jensen's inequality.

The following variant of Lebesgue's dominated convergence will be of use in the next proof and can be found in \cite[Theorem 1.20]{Evans2015}, for instance.

\begin{lemma}[Variant of dominated convergence]\label{lem:dom} Let $(X,\mathcal A, \sigma)$ be a measure space. If $(f_n)$, $(g_n)$ are sequences of measurable functions on $X$ with $|f_n|\leq g_n$ for all $n\in \N$, $f_n\to f$ and $g_n\to g$ pointwise almost everywhere, and $\int_X g_n\,\mathrm d \sigma\to\int_X g \,\mathrm d\sigma$ as $n\to\infty$, where the integrals involved are finite, then $$ \lim_{n\to\infty} \int_X |f_n -f| \,\mathrm d\sigma=0 \,.$$
\end{lemma}

\begin{proof}[Proof of Lemma \ref{lem:cpctresult}]
    By H\"older's inequality, we have
    $$
    \int_{B^d}|\phi_n|^{{q'}}\,\mathrm d \nu\leq \left(\int_{B^d}(1+\varepsilon_n|\phi_n|)^{q'-2}\phi_n^2\,\mathrm d\nu\right)^{\frac {q'}2}\left(\int_{B^d}(1+\varepsilon_n|\phi_n|)^{{q'}}\,\mathrm d\nu\right)^{1-\frac {q'}2}.
    $$
    Since the first factor is bounded by assumption \eqref{eq:bddassump} and $(1+\varepsilon_n|\phi_n|)^{{q'}}\lesssim 1+|\phi_n|^{{q'}}$, we find
    $$\|\phi_n\|^{{q'}}_{q'}\lesssim \left(1+\|\phi_n\|_{q'}^{{q'}}\right)^{1-\frac {q'} 2}\,.$$
    This forces the sequence $(\phi_n)$ to be bounded in $L^{{q'}}(B^d)$, which implies weak convergence along a subsequence to a $\phi\in L^{{q'}}(B^d)$.

    Since $S:L^{q'}(B^d)\to L^{p'}(\Sph^{d-1})$ is bounded, we have $S\phi_n\rightharpoonup S\phi$ weakly in $L^{p'}(\Sph^{d-1})$. By Schauder's theorem, since the operator $Q:L^{r}(\mathbb S^{d-1})\to L^{q}(B^d)$, $r>p$, is compact by Lemma~\ref{lem:cpctQ}, its dual $S:L^{q'}(B^d)\to L^{\frac{r}{r-1}}(\mathbb S^{d-1})$ is also compact. Therefore, $(S\phi_n)$ converges strongly in $L^{\frac{r}{r-1}}(\mathbb S^{d-1})$ for any $r>p$. It is easy to see that the limit is equal to $S\phi$. Thus, after passing to a subsequence, we can ensure that $S\phi_n\to S\phi$ pointwise almost everywhere.

    To prove the second part of the lemma, we divide $\Sph^{d-1}$ into
    $$
    A_n\coloneqq \{\varepsilon_n |S\phi_n|\leq 2\}\qquad \text{and}\qquad A^c_n=\{\varepsilon_n |S\phi_n|>2\}\,.
    $$ 
    We are going to prove, for any fixed $C>0$,
    \begin{align}\label{eq:approxL2convAn}
        \lim_{n\to\infty} \int_{A_n} \frac{(1+C\varepsilon_n|S\phi_n|)^{{p'}}}{1+\varepsilon^2_n(S\phi_n)^2} (S\phi_n)^2 \,\mathrm d\mu
        &=\int_{\mathbb S^{d-1}}(S\phi)^2\,\mathrm d\mu < \infty \,,\\\label{eq:approxL2convAnc}
        \lim_{n\to\infty} \int_{A_n^c}\frac{(1+C\varepsilon_n|S\phi_n|)^{{p'}}}{1+\varepsilon^2_n(S\phi_n)^2}(S\phi_n)^2\,\mathrm d\mu&=0\,,
    \end{align}
    which will clearly imply the assertion of the lemma.
    
    \textit{Proof of \eqref{eq:approxL2convAn}:} As a preliminary, we prove the bound
    \begin{equation}
        \label{eq:higherint0}
            \int_{A_n} |S\phi_n|^\frac{2p'}{q'}\,\mathrm d\mu  \lesssim 1 \,.
    \end{equation}
    
    To do so, we note that, for any $\varepsilon>0$, the function defined by 
    $$
    t\mapsto |t|^{\frac 2{q'} } (1+\varepsilon |t|)^{\frac {{q'}-2}{{q'}}}$$ 
    is convex on $\R$, since its second derivative is $(1+\varepsilon t)^{\frac {{q'}-2}{{q'}}-2}t^{\frac 2{q'} -2}2(2-{q'}){q'}^{-2} >0$ on $(0,\infty)$, and it is differentiable at the origin. Therefore, by applying Jensen's inequality to the integral defining the action of the operator $S$ and recalling that $S1=1$, we deduce that
    $$
    (1+\varepsilon_n |S\phi_n|)^{\frac{{q'}-2}{{q'}}} |S\phi_n|^{\frac 2{q'}}
    \leq S\left((1+\varepsilon_n|\phi_n|)^{\frac{{q'}-2}{{q'}}}|\phi_n|^{\frac 2{q'}} \right).
    $$
    Since
    $$
    1\leq 3^{\frac{2-q'}{q'}}(1+\varepsilon_n|S\phi_n|)^{\frac{{q'}-2}{{q'}}}
    \qquad\text{on}\ A_n \,,
    $$
    we deduce from the dual HWY-inequality that
    \begin{align*}
        \int_{A_n}|S&\phi_n|^{2\frac {p'}{q'}}\,\mathrm d\mu
        \lesssim \int_{A_n}\left((1+\varepsilon_n|S\phi_n|)^{\frac{{q'}-2}{{q'}}}|S\phi_n|^{\frac 2{q'}}\right)^{p'}\mathrm d\mu 
        \leq \int_{A_n}\left(S\left((1+\varepsilon_n|\phi_n|)^{\frac{{q'}-2}{{q'}}}|\phi_n|^{\frac 2{q'}}\right)\right)^{p'}\mathrm d\mu\\        &\lesssim\left(\int_{B^d}\left((1+\varepsilon_n|\phi_n|)^{\frac{{q'}-2}{{q'}}}|\phi_n|^{\frac 2{q'}}\right)^{q'}\mathrm d\nu \right)^{\frac{p'}{q'}}= \left(\int_{B^d}(1+\varepsilon_n|\phi_n|)^{{q'}-2}\phi_n^2\,\mathrm d\nu \right)^{\frac{p'}{q'}}\leq 1\,,
    \end{align*} 
    where we used assumption \eqref{eq:bddassump} in the last step. This completes the proof of the bound \eqref{eq:higherint0}.
    
    As a consequence of \eqref{eq:higherint0}, we deduce the higher integrability result
    \begin{equation}
        \label{eq:higherint}
            S\phi \in L^{2\frac{p'}{q'}}(\Sph^{d-1}) \,.
    \end{equation}
    Indeed, since $|S\phi_n|^{2\frac {p'}{q'}} \1_{A_n} \to |S\phi|^{2\frac {p'}{q'}}$ almost everywhere, this is a consequence of \eqref{eq:higherint0} and Fatou's lemma.

    Next, we deduce that
    \begin{equation}
        \label{eq:convimproved}
            S\phi_n\mathbbm 1_{A_n} \to S\phi
            \qquad\text{in}\ L^s(\Sph^{d-1}) \ \text{for all}\ s<2\,\frac{p'}{q'} \,.
    \end{equation}
    Indeed, by H\"older's inequality, there is a $\theta\in(0,1)$ such that
    $$
    \|S\phi_n\mathbbm 1_{A_n}-S\phi\|_{s}\leq \|S\phi_n\mathbbm 1_{A_n}-S\phi\|_{2p'/q'}^\theta \|S\phi_n\mathbbm 1_{A_n}-S\phi\|_{1}^{1-\theta}\,,
    $$
    where the first factor on the right remains bounded by \eqref{eq:higherint0} and \eqref{eq:higherint}, while the second factor tends to zero. (Here we use $\mu(A_n^c)\to 0$ by $\epsilon_n S\phi_n\to 0$ almost everywhere and dominated convergence.)
    
    After these preliminaries, we now fix $C>0$ and turn to the proof of \eqref{eq:approxL2convAn}. We apply the variant of the dominated convergence theorem in Lemma \ref{lem:dom} with $$f_n\coloneqq \frac{(1+C\epsilon_n|S\phi_n|)^{p'}}{1+\varepsilon_n^2 (S\phi_n)^2} (S\phi_n)^2 \1_{A_n}\,,$$
    which converges almost everywhere to $f\coloneqq (S\phi)^2$. Denoting $m\coloneqq  \sup_{t\in[0,2]}(1+Ct)^{p'}/(1+t^2)$, we see that $|f_n|\leq m (S\phi_n)^2\1_{A_n}=:g_n$, which converges almost everywhere to $m(S\phi)^2 =: g$. The fact that $\int_{\Sph^{d-1}} g_n\,\mathrm d\mu \to \int_{\Sph^{d-1}} g\,\mathrm d\mu$ follows from \eqref{eq:convimproved} with $s=2$. (Note that $2p'/q'>2$.) Thus, Lemma \ref{lem:dom} implies that $\int_{\Sph^{d-1}} f_n\,\mathrm d\mu \to \int_{\Sph^{d-1}} f\,\mathrm d\mu$, which is what is claimed in \eqref{eq:approxL2convAn}.

\textit{Proof of \eqref{eq:approxL2convAnc}:}
First, we estimate 
\begin{equation}\label{eq:firstbdd}\int_{A_n^c}\frac{(1+C\varepsilon_n|S\phi_n|)^{{p'}}}{1+\varepsilon^2_n(S\phi_n)^2}(S\phi_n)^2\,\mathrm d\mu\lesssim \int_{A_n^c}(1+\varepsilon_n|S\phi_n|)^{{p'}-2}(S\phi_n)^2\,\mathrm d\mu\leq \varepsilon_n^{{p'}-2}\int_{A_n^c}|S\phi_n|^{p'}\,\mathrm d\mu\,.
\end{equation}
We define 
$$
\phi_{n,1}\coloneqq \phi_n\mathbbm 1_{\{\varepsilon_n |\phi_n|\leq 1\}}
\qquad\text{and}\qquad
\phi_{n,2}\coloneqq\phi_n\mathbbm 1_{\{\varepsilon_n |\phi_n|> 1\}} \,.
$$
Since
$\|Q(\,\cdot\,,\omega)\|_{1}=1$, $\omega\in \mathbb S^{d-1}$, we obtain 
\begin{equation*}
    \varepsilon_n\|S\phi_{n,1}\|_\infty \leq \varepsilon_n\|\phi_{n,1}\|_{\infty}\leq 1\,,
\end{equation*}
which implies $\varepsilon |S\phi_{n,2}|>1$ on $A_n^c$ by triangle inequality. 
Applying the dual HWY-inequality, we find
\begin{align}\notag\varepsilon_n^{{p'}-2}\int_{A_n^c}|S\phi_{n,2}|^{p'}\,\mathrm d\mu&\lesssim \varepsilon_n^{{p'}-2}\left(\int_{\{\varepsilon_n|\phi_n|>1\}}|\phi_{n}|^{q'}\,\mathrm d\nu\right)^{\frac{p'}{q'}}\\&\lesssim 
\varepsilon_n^{2\frac{{p'}-{q'}}{q'}}\left(\int_{B^d}(1+\varepsilon_n|\phi_{n}|)^{{q'}-2}\phi_{n}^2\,\mathrm d\nu\right)^{\frac{p'}{q'}}\lesssim \varepsilon_n^{2\frac{{p'}-{q'}}{q'}},\label{eq:boundlarger1}
\end{align} 
which tends to $0$ as $n\to\infty$. In the last step, we used assumption \eqref{eq:bddassump}. Since $\varepsilon |S\phi_{n}|\leq1+\varepsilon |S\phi_{n,2}|< 2\varepsilon |S\phi_{n,2}|$ on $A_n^c$, combining the estimates \eqref{eq:firstbdd} and \eqref{eq:boundlarger1} implies \eqref{eq:approxL2convAnc}.
 \end{proof}


\begin{remark}[A bound of Orlicz type] 
    Note that a similar but simpler proof than that of Lemma~\ref{lem:cpctresult} yields the following assertion: For any $\epsilon_0>0$ there is a $C=C(d,\epsilon_0)>0$ such that for all $\varepsilon\in (0,\varepsilon_0]$ and $\phi\in  L^{q'}(B^d)$ satisfying
    \begin{equation*}
        \int_{B^d} \frac{\phi^2}{(1+\varepsilon|\phi|)^{2-q'}} \,\mathrm d \nu\leq 1\,,
    \end{equation*}
    we have 
    \begin{equation*}
    \int_{\Sph^{d-1}} \frac{(S\phi)^2}{(1+\varepsilon|S\phi|)^{2-p'}} \,\mathrm d\mu\leq C\int_{B^d} \frac{\phi^2}{(1+\varepsilon|\phi|)^{2-q'}} \,\mathrm d \nu \,.
    \end{equation*}
    This Orlicz-type bound for the dual of the harmonic extension operator might be of independent interest. In \cite[Corollary 3.5]{Figalli2022} an analogous bound is proved for the $p$-Sobolev inequality. In our proof of Proposition \ref{prop:localstabilitydual} such an inequality is not needed though.
\end{remark}


\subsection{Spectral gaps} 
\label{subsec:specgap2}
As in Subsection \ref{subsec:specgap}, the goal of this subsection is to prove a `non-linear spectral gap' inequality -- this time for the dual HWY-inequality \eqref{eq:HWYspheredual}. We recall that the quantity $\zeta$ was introduced in Lemma \ref{lem:elemineq2} and that the space spanned by spherical harmonics of degree $0$ and $1$ is denoted by $\mathcal H^l$.

\begin{proposition}[Non-linear spectral gap inequality for $S$]\label{prop:anospecV}
 Let $d\geq 3$ and $\lambda\in (0,2)$. For any $\gamma_0>0$ and $C>0$, there is a $\delta=\delta(d, \lambda, \gamma_0,C)>0$ such that for any $\phi\in L^{{q'}}(B^d)\cap (Q\mathcal H^l)^\perp$ with $\|\phi\|_{q'}\leq \delta$ we have
 \begin{align*}
    \int_{B^{d}} (\phi^2+&(q'-2)\zeta(\phi)(1-|1+\phi|)^2)\, \mathrm d\nu +\gamma_0\int_{B^d}\min\{|\phi|^{{q'}},\phi^{2}\}\, \mathrm d\nu \\&\geq (q'-1)\frac{d+2+\lambda}{d}\int_{\mathbb S^{d-1}}\frac{(1+C|S\phi|)^{p'}}{1+(S\phi)^2}(S\phi)^2\,\mathrm d\mu\,.
 \end{align*} 
 \end{proposition}

The proof of Proposition \ref{prop:anospecV} proceeds by a compactness argument (namely, Lemma \ref{lem:cpctresult}) to reduce matters to a linear spectral inequality for the operator $S$. The latter can be deduced via duality from the corresponding linear spectral gap inequality for $Q$ in Lemma \ref{lem:spec}.

\begin{lemma}[Spectral gap inequality for $S$]\label{lem:specV}
 Let $d\geq 3$. For any function $\phi\in L^{2}(B^d)\cap (Q\mathcal H^l)^\bot$, we have
 \begin{equation*}
 \|\phi\|_{2}^2\geq \frac{d+4}{d}\|S\phi\|_2^2\,.\end{equation*}
 \end{lemma}

\begin{proof}
    Let $\psi\in L^2(\Sph^{d-1})$ and let $\Pi^h\psi$ be the projection of $\psi$ onto the orthogonal complement of $\mathcal H^l$. Then for $\phi\bot Q\mathcal H^l$,
    $$
    \langle \psi, S\phi \rangle = \langle \Pi^h \psi, S \phi \rangle = \langle Q \Pi^h\psi,\phi \rangle \,.
    $$
    Thus, by Lemma \ref{lem:spec},
    $$
    \left| \langle \psi, S\phi \rangle \right| \leq \|Q \Pi^h \psi\|_{2} \|\phi\|_{2} \leq \sqrt\frac{d}{d+4} \|\Pi^h \psi \|_2 \|\phi\|_{2} \leq \sqrt\frac{d}{d+4} \|\psi \|_2 \|\phi\|_{2} \,.
    $$
    Choosing $\psi=S\phi$, we obtain the claimed inequality.
\end{proof}

\begin{proof}[Proof of Proposition \ref{prop:anospecV}]
    By contradiction, assume that there is a non-trivial sequence $(\phi_n)\subset L^{q'}(B^d)\cap (Q\mathcal H^l)^\perp$
    with $\phi_n\rightarrow 0$ 
    in $L^{q'}(B^d)$ as $n\to\infty$ and
    \begin{equation}\label{eq:Vspec1}
    \begin{split}
    \int_{B^{d}} (\phi_n^2+ &(q'-2)\zeta(\phi_n)(1-|1+\phi_n|)^2)\, \mathrm d\nu +\gamma_0\int_{B^d}\min\{|\phi_n|^{{q'}},\phi_n^2\}\, \mathrm d\nu \\
    &< (q'-1)\frac{d+2+\lambda}{d}\int_{\mathbb S^{d-1}}\frac{(1+C|S\phi_n|)^{p'}}{1+(S\phi_n)^2}(S\phi_n)^2\,\mathrm d\mu\,.
     \end{split}
    \end{equation} 
    For
    \begin{equation*}
    \varepsilon_n\coloneqq\left(\int_{B^d}(1+|\phi_n|)^{q'-2}\phi_n^2 \,\mathrm d \nu \right)^{1/2}
    \end{equation*}
    we set $\hat{\phi}_n\coloneqq \varepsilon_n^{-1}\phi_n$. As $q'-2<0$, it follows that
    \begin{equation*}
        \varepsilon_n^2\leq  \|\phi_n\|_{q'}^{q'} \rightarrow 0\,.
    \end{equation*}

    We note that the integrand of the first integral on the left side of \eqref{eq:Vspec1} is pointwise nonnegative. Hence, to obtain a lower bound, we can restrict the domain of integration.
    
    Fix $\tilde\varepsilon>0$. We choose $\tilde\delta\in (0,1)$ sufficiently small so that $\zeta(\phi_n)\leq 1+\frac{q'-1}{2-q'}\tilde\varepsilon$ on $I_n\coloneqq\{|\phi_n|\leq \tilde \delta\}$. On $I_n$ we also have $1+\phi_n\geq 0$, and so it follows that
    \begin{equation*}
       \phi_n^2+(q'-2)\zeta(\phi_n)\left( 1- |{\phi}_n+1| \right)^2\geq  (q'-1)(1-\tilde\varepsilon)\phi_n^2\,.
    \end{equation*}
    For the second term on the left side of \eqref{eq:Vspec1}, we restrict to $I_n^c=\{|\phi_n|>\tilde\delta\}$. In this way, we obtain from \eqref{eq:Vspec1} that
    \begin{equation}\label{eq:Vspec2}
    \begin{split}
    (1-\tilde\varepsilon)\int_{I_n} \hat\phi_n^2 \, \mathrm d\nu + \frac{\gamma_0}{q'-1} \epsilon_n^{q'-2}{\tilde\delta}^{2-q'} \int_{I_n^c} |\hat \phi_n|^{{q'}} \, \mathrm d\nu
    < \frac{d+2+\lambda}{d}\int_{\mathbb S^{d-1}}\frac{(1+C|S\phi_n|)^{p'}}{1+(S\phi_n)^2}(S\hat \phi_n)^2\,\mathrm d\mu\,.
     \end{split}
    \end{equation} 

    By definition, we have $\int_{B^d} (1+\epsilon_n|\hat\phi_n|)^{q'-2} \hat\phi_n^2\,d\nu = 1$, and therefore, by Lemma \ref{lem:cpctresult}, after passing to a subsequence, we have that $\hat{\phi}_n\rightharpoonup \hat{\phi}$ in $L^{q'}(B^d)$ for some $\hat{\phi}\in L^{q'}(B^d)$ with $S\hat\phi\in L^2(\Sph^{d-1})$, and 
    \begin{equation*}
        \int_{\Sph^{d-1}}\frac{(1+C|S\phi_n|)^{p'}}{1+(S\phi_n)^2}(S\hat{\phi}_n)^2\,\mathrm d \mu \rightarrow \int_{\Sph^{d-1}}(S\hat{\phi})^2\,\mathrm d\mu\,.
    \end{equation*}

    Combining this with 
    \eqref{eq:Vspec2}, we infer
    \begin{equation}
        \label{eq:Vspec5}
        \limsup_{n\to\infty} \left( (1-\tilde\varepsilon)\int_{I_n} \hat\phi_n^2\,\mathrm d\nu + \frac{\gamma_0}{q'-1}\tilde\delta^{2-q'} \epsilon_n^{q'-2} \int_{I_n^c} |\hat \phi_n|^{{q'}} \, \mathrm d\nu \right)
        \leq \frac{d+2+\lambda}{d} \int_{\Sph^{d-1}} (S\hat\phi)^2\,\mathrm d\mu \,.
    \end{equation}

    A first consequence of \eqref{eq:Vspec5} is that $S\hat\phi\neq 0$. Indeed, we have
    \begin{equation*}
        1= \int_{B^d}(1+|\phi_n|)^{q'-2}\hat\phi_n^2 \,\mathrm d \nu \leq \int_{I_n}|\hat{\phi}_n|^2\,\mathrm d \nu+\int_{I_n^c}\varepsilon^{q'-2}_n|\hat{\phi}_n|^{q'}\,\mathrm d \nu\,,
    \end{equation*}
    and, according to \eqref{eq:Vspec5}, the limsup of the right side is bounded by a constant times $\|S\hat\phi\|_2^2$, so the latter quantity is non-zero, as claimed. We mention in passing that this part of the argument is the only part where $\gamma_0>0$ is used.

    Another consequence of \eqref{eq:Vspec5} is that the sequence $(\hat\phi_n \1_{I_n})$ is bounded in $L^2(B^d)$ and therefore has a weak limit point. We claim that
    \begin{equation}\label{eq:Vspec6}
        \hat\phi_n \1_{I_n} \rightharpoonup \hat\phi
        \qquad\text{weakly in}\ L^2(B^d) \,.
    \end{equation}
    To prove this, let $\tilde\phi\in L^2(B^d)$ be a weak limit point of $(\hat\phi_n\1_{I_n})$ in $L^2(B^d)$. Thus, along a subsequence, we obtain for all $f\in L^2(B^d)$ that
    $$
    \int_{I_n} f\hat\phi_n\,\mathrm d\nu \to \int_{B^d} f \tilde\phi\,\mathrm d\nu \,.
    $$
    Under the stronger assumption $f\in L^q(B^d)$, we can decompose the left side as
    $$
    \int_{I_n} f\hat\phi_n\,\mathrm d\nu = \int_{B^d} f\hat\phi_n\,\mathrm d\nu - \int_{I_n^c} f\hat\phi_n\,\mathrm d\nu \,,
    $$
    where the first term converges to $\int_{B^d} f\hat\phi\,\mathrm d\nu$. The second term is bounded by
    $$
    \left| \int_{I_n^c} f\hat\phi_n\,\mathrm d\nu \right| \leq \| f\1_{I_n^c}\|_{q} \|\hat\phi_n \|_{q'} \,.
    $$
    This upper bound tends to zero since $\|\hat\phi_n\|_{q'}$ is bounded and since
    $$
    \nu(I_n^c) \leq \tilde\delta^{-q'}\int_{I_n^c} |\phi_n|^{q'}\,\mathrm d\nu \leq \tilde\delta^{-q'}\|\phi_n\|_{q'}^{q'} \to 0 \,,
    $$
    so $\|f\1_{I_n^c}\|_{q}\to 0$. Thus, we obtain
    $$
    \int_{B^d} f \tilde\phi\,\mathrm d\nu = \int_{B^d} f\hat\phi\,\mathrm d\nu
    \qquad\text{for all}\ f\in L^q(B^d) \,.
    $$
    Since $L^q(B^d)$ is dense in $L^2(B^d)$, this implies $\tilde\phi=\hat\phi$, as claimed in \eqref{eq:Vspec6}.

    In view of weak lower semicontinuity, \eqref{eq:Vspec5} and \eqref{eq:Vspec6} imply
    $$
    (1-\tilde\varepsilon)\|\hat\phi\|_{2}^2 \leq \frac{d+2+\lambda}{d} \|S\hat\phi\|_2^2 \,.
    $$
    Since this holds for any $\tilde\varepsilon>0$, we can take $\tilde\varepsilon= 0$. 
    The orthogonality condition $\hat\phi_n\bot Q\mathcal H^l$ implies $\hat\phi\bot Q\mathcal H^l$. Therefore, recalling also $\lambda<2$ and $S\hat\phi\neq 0$, the previous inequality contradicts the spectral gap inequality in Lemma \ref{lem:specV}, which concludes the proof.
\end{proof}


\subsection{Close optimizer with orthogonality conditions}

In this subsection we show that functions close to the set of optimizers can be M\"obius transformed in order to achieve orthogonality conditions while, at the same time, maintaining some closeness to the set of optimizers. While we proved a similar result for the primal HWY-inequality in Proposition \ref{prop:almorthog}, the details in the dual setting are rather different.

\begin{proposition}[Orthogonal approximate optimizer]
\label{prop:dualorthogmin}
   Let $(v_n)\subset L^{q'}(B^d)$ be a sequence of functions with
   $$
   \| v_n \|_{q'}=1
   \qquad\text{and}\qquad
   \inf_\Phi \| [v_n]_\Phi - 1\|_{q'}\to 0
   $$
   for $n\to\infty$. Then there is a sequence of M\"obius transformations $\Phi_n$ such that
   $$
   r_n\coloneqq [v_n]_{\Phi_n} - 1
   $$
   satisfies, as $n\to\infty$,
   $$
   \|r_n\|_{q'} \to 0
   $$
   and, for all sufficiently large $n$,
   $$
   \int_{B^d} y_i r_n(y)\,d\nu(y) = 0
   \qquad\text{for all}\ i=1,\ldots,d\,.
   $$
\end{proposition}


There are two differences compared to Proposition \ref{prop:almorthog}. First, instead of \emph{almost} orthogonality conditions as used in the latter, here we achieve genuine orthogonality conditions. Second, instead of $d+1$ (almost) orthogonality conditions in the primal setting, here we only receive $d$. (The remaining orthogonality condition will be achieved in the proof of Proposition \ref{prop:localstabilitydual}.)

The argument we give is based on \cite[Lemma 4.1]{Figalli2022}.

\begin{proof}[Proof of Proposition \ref{prop:dualorthogmin}]
    \emph{Step 1.}
   We recall the explicit family $\Phi_\eta$, $\eta\in B^d$, of M\"obius transformations of $B^d$ given in \eqref{eq:phieta}. For $v\in L^{q'}(B^d)$ we define the functional $F_v:B^d\to\R$,
    \begin{equation*}
        {F}_v(\eta)\coloneqq\int_{B^d}[v]_{\Phi_{\eta}}\,\mathrm d \nu \,. 
    \end{equation*}
    Note that, by H\"older's inequality and conformal invariance,
    \begin{equation}
        \label{eq:orthodualproof}
            |F_v(\eta)| \leq \|[v]_{\Phi_\eta}\|_{q'} = \|v\|_{q'} \,.
    \end{equation}
    
    We will be interested in maximizing $F_v$. We first note that if $v=1$, then $\eta=0$ is the unique maximizer. Indeed, for $v=1$ the H\"older inequality in \eqref{eq:orthodualproof} is saturated if and only if $[1]_{\Phi_\eta}$ is constant. According to \eqref{eq:phietajac}, this is equivalent to $\eta=0$.

    If more generally $\|v-1\|_{q'}<1$, we will argue that ${F}_v$ attains its maximum in $B^d$. Indeed, on the one hand, we have 
    $$
    \sup_{\eta \in B^d}{F}_v(\eta)\geq {F}_{v}(0)={F}_{1}(0)-( F_{1}(0)- F_{v}(0)) \geq 1-\|v-1\|_{q'}\,,
    $$ 
    and the lower bound on the right side is, by assumption, a positive number. On the other hand, 
    \begin{equation}\label{eq:orthodualproof2}
       {F}_{v}(\eta)\leq  \|[v]_{\Phi_{\eta}}\|_{1} \to 0
       \qquad\text{as}\ |\eta|\to 1 \,;
    \end{equation}
    see \cite[Lemma 7]{Frank2024B}, for instance. By continuity of the function $F_v$ and by compactness, $\sup_{B^d} F_v$ is attained in $B^d$. 

    \medskip

    \emph{Step 2.}    
    Consider now a sequence $(v_n)\subset L^{q'}(B^d)$ with $\| v_n \|_{q'}=1$ and $\inf_\Phi \| [v_n]_\Phi - 1\|_{q'}\to 0$. We consider $n$ so large that the latter infimum is smaller than $1$. We apply Step 1 with $v$ replaced by $\tilde v_n \coloneqq  [v_n]_{\Phi_n'}$, where $\Phi_n'$ is a M\"obius transformation such that $\|[v_n]_{\Phi_n'}-1\|_{q'}<1$ and such that $\|[v_n]_{\Phi_n'}-1\|_{q'}\to 0$ as $n\to\infty$. We obtain a sequence $(\eta_n)\subset B^d$ such that $\sup_{B^d} F_{\tilde v_n}$ is attained at $\eta_n$. Set $\Phi_n\coloneqq  \Phi_n'\circ\Phi_{\eta_n}$ and $r_n\coloneqq [v_n]_{\Phi_n}-1$.

    Let us show that $\eta_n\to 0$. On the one hand, we have, by optimality,
    $$
    F_1(\eta_n) \leq F_1(0) \,.
    $$
    On the other hand, we have, by H\"older's inequality, conformal invariance, and optimality,
    $$
    F_1(\eta_n) \geq F_{\tilde v_n}(\eta_n) - \| \tilde v_n - 1\|_{q'} \geq F_{\tilde v_n}(0) - \| \tilde v_n - 1\|_{q'} \geq F_1(0) - 2 \| \tilde v_n - 1\|_{q'} \,.
    $$
    It follows that $F_1(\eta_n)\to F_1(0)$ as $n\to\infty$. To prove that $\eta_n\to 0$, let $\eta_\infty$ denote any limit point of $(\eta_n)$. Thus, along an appropriate subsequence, we have $F_1(\eta_n)\to c$, where $c\coloneqq F_1(\eta_\infty)$ if $|\eta_\infty|<1$ and $c\coloneqq 0$ if $|\eta_\infty|=1$ by \eqref{eq:orthodualproof2}. Now $F_1(\eta_n)\to F_1(0)=1$ implies $c=1$, and therefore, since $\eta=0$ uniquely maximizes $F_1$, we have $\eta_\infty=0$, as claimed.

    Since $\eta_n\to 0$, we conclude that, as $n\to\infty$,
    $$
    \| r_n \|_{q'} = \| \tilde v_n - [1]_{\Phi_{\eta_n}^{-1}} \|_{q'} \leq \| \tilde v_n - 1 \|_{q'} + \| [1]_{\Phi_{\eta_n}^{-1}} -1 \|_{q'} \to 0 \,.
    $$ 
    
    Finally, we prove the orthogonality conditions. Since $\eta_n$ is the maximizer of ${F}_{\tilde v_n}$, we have
    $$
    \int_{B^d} (1+r_n)\,\mathrm d\nu \geq \int_{B^d} [1+r_n]_{\Phi_\eta} \,\mathrm d\nu
    \qquad\text{for all}\ \eta\in B^d \,.
    $$
    Thus, the gradient of $\eta\mapsto \int_{B^d} [1+r_n]_{\Phi_\eta} \,\mathrm d\nu$ vanishes at $\eta=0$. Using $\int_{B^d} [1+r_n]_{\Phi_\eta} \,\mathrm d\nu = \int_{B^d} [1]_{\Phi_\eta^{-1}}^{q'-1} (1+r_n) \,\mathrm d\nu$ and the explicit expression for $[1]_{\Phi_\eta^{-1}}$ in \eqref{eq:phietajac} (note $\Phi_\eta^{-1}=\Phi_{-\eta}$), we easily find that
    $$
    \nabla_\eta \Big\vert_{\eta=0} \int_{B^d} [1+r_n]_{\Phi_\eta} \,\mathrm d\nu = - (d-2) \int_{B^d} y \,r_n(y)\,\mathrm d\nu(y)  \,.
    $$
   This implies the claimed orthogonality conditions.
\end{proof}

\subsection{Local analysis}
In this subsection we combine the ingredients of the previous two subsections and finish the proof of the local bound, Proposition \ref{prop:localstabilitydual}. 
\begin{proof}[Proof of Proposition \ref{prop:localstabilitydual}] 
    We start with a sequence $(v_n) \subset L^{q'}(B^d)$ satisfying $\|v_n\|_{q'}=1$ for all $n$ and $\inf_\Phi\|[v_n]_\Phi-1\|_{q'}\to 0$ as $n\rightarrow \infty$. Proposition \ref{prop:dualorthogmin} gives us a sequence $(\Phi_n)$ of M\"obius transformations such that $\|[v_n]_{\Phi_n}-1\|_{q'}\to 0$ and such that $[v_n]_{\Phi_n}-1$ is $L^2(B^d)$-orthogonal to $\operatorname{span}\{y_1,\ldots,y_d\}$.

    Set $\alpha_n\coloneqq  \int_{B^d}[v_n]_{\Phi_n}\mathrm d\nu$ and note that, by H\"older's inequality,
    $$
    \left|\alpha_n-1\right| = \left| \int_{B^d}([v_n]_{\Phi_n}-1)\,\mathrm d\nu\right|\leq \|[v_n]_{\Phi_n}-1\|_{q'} \,.
    $$
    Since the right side tends to zero, we infer, in particular, that $\alpha_n\geq 0$ for all sufficiently large $n$ and, consequently, by the normalization $\|[v_n]_{\Phi_n}\|_{q'}=\|v_n\|_{q'}=1$,
    $$
    \left|\alpha_n -1\right| = \left||\alpha_n|-\|[v_n]_{\Phi_n}\|_{q'}\right| \leq \left\|[v_n]_{\Phi_n}-\alpha_n \right\|_{q'} \,.
    $$
    Thus, we have shown, for all sufficiently large $n$,
    \begin{equation*}
        \left|\alpha_n-1\right|\leq \min\left\{\|[v_n]_{\Phi_n}-1\|_{q'},\left\|[v_n]_{\Phi_n}-\alpha_n\right\|_{q'}\right\},
    \end{equation*}
    which, in turn, implies
    \begin{equation}
    \label{eq:orthodistapprox} 
        \frac 12\|[v_n]_{\Phi_n}-1\|_{q'}\leq \left\|[v_n]_{\Phi_n}-\alpha_n\right\|_{q'} \leq 2\|[v_n]_{\Phi_n}-1\|_{q'} \,.
    \end{equation}

    We set $\tilde r_n\coloneqq \alpha_n^{-1}[v_n]_{\Phi_n} - 1$. According to Proposition \ref{prop:dualorthogmin} and \eqref{eq:orthodistapprox}, we have $\|\tilde r_n\|_{q'}\to 0$ and $\tilde r_n$ is $L^2(B^d)$-orthogonal to $\operatorname{span}\{1,y_1,\ldots,y_d\}=Q\mathcal H^l$. Moreover, by conformal invariance, we have
    $$
    1 - \|Sv_n\|_{p'}^{q'} = \alpha_n^{q'} \left( \|1+ \tilde r_n\|_{q'}^{{q'}}-\|S(1+ \tilde r_n)\|_{p'}^{{q'}} \right),
    $$
    and our task is to find a lower bound on the right side.

    We use the elementary estimates from Lemma \ref{lem:elemineq2}. More specifically, for any $\kappa>0$ there are $C_\kappa,c_\kappa>0$ such that for all $r\in L^{q'}(B^d)$ we have, with $\zeta(r)$ from Lemma \ref{lem:elemineq2},
    \begin{align}
        \label{eq:Nbounddual}
        \|1+r\|_{q'}^{q'}&\geq 1+{q'}\int_{B^d}r\,\mathrm d \nu+q'\frac{1-\kappa}{2}\int_{B^d}\left(r^2+({q'}-2)\zeta(r)(|r+1|-1)^2\right)\mathrm d \nu \notag \\
        & \quad +c_\kappa\int_{B^d}\min \{|r|^{q'},r^2\}\,\mathrm d \nu
        \intertext{and}
        \|S(1+ r)\|_{p'}^{{p'}}&\leq 1+{p'}\int_{\Sph^{d-1}} Sr\,\mathrm d \mu+\left(\frac{{p'}({p'}-1)}{2}+\kappa\right)\int_{\Sph^{d-1}}\frac{(1+C_\kappa|Sr|)^{{p'}}}{1+(Sr)^2}(Sr)^2\,\mathrm d \mu \,. \notag
    \end{align}
    Moreover, using the elementary inequality $(1+a)^{{q'}/{p'}}\leq 1+a{q'}/{p'}$ for $a\geq-1$, it follows that
    \begin{align}\label{eq:FZbounddual}
         \|S(1+ r)\|_{p'}^{{q'}}&\leq 1+{q'}\int_{\Sph^{d-1}} Sr\,\mathrm d \mu+\left(\frac{{q'}({p'}-1)}{2}+\frac{{q'}\kappa}{{p'}}\right)\int_{\Sph^{d-1}}\frac{(1+C_\kappa|Sr|)^{{p'}}}{1+(Sr)^2}(Sr)^2\,\mathrm d \mu\,.
    \end{align}

    Applying \eqref{eq:Nbounddual} and \eqref{eq:FZbounddual} to $r=\tilde r_n$ and noting that $\int_{B^d} \tilde r_n\,\mathrm d\nu = \int_{\Sph^{d-1}} S\tilde r_n\,\mathrm d\mu$, we obtain
    \begin{align*}
        \|1 + \tilde r_n\|_{q'}^{q'} - \|S(1+\tilde r_n)\|_{p'}^{q'} & \geq \frac{q'}{2}(1-\kappa)\int_{B^d} \left( \tilde r_n^2 + ({q'}-2)\zeta(\tilde r_n)(|\tilde r_n+1|-1)^2\right) \mathrm d \nu \\
        & \quad - \left(\frac{{q'}({p'}-1)}{2}+\frac{{q'}\kappa}{{p'}}\right)\int_{\Sph^{d-1}}\frac{(1+C_\kappa|S\tilde r_n|)^{{p'}}}{1+(S\tilde r_n)^2}(S\tilde r_n)^2\,\mathrm d \mu \\
        &\quad +c_\kappa\int_{B^d}\min \{|\tilde r_n|^{q'},\tilde r_n^2\}\,\mathrm d \nu \,.
    \end{align*}
    The right side can be bounded from below via the non-linear spectral gap inequality, Proposition~\ref{prop:anospecV}. Given any $\lambda<2$, $\gamma_0>0$, and $C_\kappa>0$, we find that for all sufficiently large $n$ (recall that $\|\tilde r_n\|_{q'}\to 0$ and $\tilde r_n\in (Q\mathcal H^l)^\bot$), we have
    \begin{align*}
        & \|1+ \tilde r_n\|_{q'}^{{q'}}-\|S(1+ \tilde r_n)\|_{p'}^{{q'}} \\
        & \quad \geq  
         \frac{q'}2 \left( (1-\kappa) - \frac{p'-1}{q'-1} \frac{d}{d+2+\lambda} \left(1+\frac{2\kappa}{p'(p'-1)}\right)
         \right) \\
         & \quad\quad \times \int_{B^d} \left( \tilde r_n^2 + ({q'}-2)\zeta(\tilde r_n)(|\tilde r_n+1|-1)^2\right) \mathrm d \nu \\
        & \quad +\left(c_\kappa -\gamma_0 \frac{q'}{2}\frac{p'-1}{q'-1} \frac{d}{d+2+\lambda} \left(1+\frac{2\kappa}{p'(p'-1)} \right)\right)\int_{B^d}\min\{|\tilde r_n|^{q'},\tilde r_n^2\}\,\mathrm d\nu \,.
    \end{align*}
    Noting that $\frac{p'-1}{q'-1}= \frac{d+2}{d}$, we see that for any $\lambda\in(0,2)$ we can choose $\kappa=\kappa(d,\lambda)>0$ small enough such that 
    \begin{equation*}
        (1-\kappa) - \frac{p'-1}{q'-1} \frac{d}{d+2+\lambda} \left(1+\frac{2\kappa}{p'(p'-1)}\right) = (1-\kappa) - \frac{d+2}{d+2+\lambda} \left(1+\frac{2\kappa}{p'(p'-1)}\right) >0 \,.
    \end{equation*}
    Having fixed $\kappa$, and therefore also $c_\kappa$, we can choose $\gamma_0>0$ small enough such that
    \begin{equation*}
        c_\kappa -\gamma_0 \frac{q'}{2}\frac{p'-1}{q'-1} \frac{d}{d+2+\lambda} \left(1+\frac{2\kappa}{p'(p'-1)} \right)
        = c_\kappa -\gamma_0 \frac{d}{d+2+\lambda}  \left(1+\frac{2\kappa}{p'(p'-1)} \right) >0\,.
    \end{equation*}
    Since $\tilde r_n^2+(q'-2)\zeta(\tilde r_n)(|\tilde r_n+1|-1)^2\geq 0$, it follows that
    \begin{equation}
        \|1+ \tilde r_n\|_{q'}^{{q'}}-\|S(1+ \tilde r_n)\|_{p'}^{{q'}} \gtrsim  \int_{B^d}\min\{|\tilde r_n|^{q'},\tilde r_n^2\}\,\mathrm d\nu\,. \label{eq:stabbdd}
    \end{equation}
    Now, using $\|\tilde r_n\|_{q'}\leq 1$ for sufficiently large $n$, we have
   \begin{align}\notag
        \int_{B^d}\min\{|&\tilde r_n|^{q'},\tilde r_n^2\}\,\mathrm d\nu =\int_{\{|\tilde r_n|\leq 1\}}\tilde r_n^2\,\mathrm d\nu+\int_{\{|\tilde r_n|> 1\}}|\tilde r_n|^{q'}\,\mathrm d \nu\\
        &\geq \nu(B^d)^{1-\frac{2}{q'}}\left(\int_{\{|\tilde r_n|\leq 1\}}|\tilde r_n|^{q'}\,\mathrm d\nu\right)^\frac{2}{q'}+\int_{\{|\tilde r_n|> 1\}}|\tilde r_n|^{q'}\,\mathrm d\nu\gtrsim\left(\int_{B^d}|\tilde r_n|^{q'}\,\mathrm d\nu\right)^\frac{2}{q'}.
        \label{eq:mintoq'}
    \end{align} 

    To summarize, we have shown that
    $$
    1 - \|Sv_n\|_{p'}^{q'} = \alpha_n^{q'} \left( \|1+ \tilde r_n\|_{q'}^{{q'}}-\|S(1+ \tilde r_n)\|_{p'}^{{q'}} \right) \gtrsim \alpha_n^{q'} \|\tilde r_n\|_{q'}^2 \,.
    $$
    Since $\alpha_n\to 1$, it follows from \eqref{eq:orthodistapprox} that the right side is comparable to $\|[v_n]_{\Phi_n} - 1\|_{q'}^2$, which, in turn, is $\geq \inf_\Phi \|[v_n]_\Phi - 1\|_{q'}^2$. This concludes the proof of Proposition \ref{prop:localstabilitydual}.
\end{proof}
\section{Optimality of the results}
\label{sec:4}

Our goal in this section is to show that the stability exponents in our main theorems are best possible.

The result in Proposition \ref{prop:localstability}, and thus in Theorems \ref{thm:HWYstabilitysphere} and \ref{thm:HWYstability}, is optimal in the sense that the exponents $2$ of the $L^2$-distance and $p$ of the $L^{p}$-distance cannot be decreased. To be more specific, for given $r\in \{p,2\}$, we can find a sequence $(u_j)\subset L^{p}(\mathbb S^{d-1})$ with $\|u_j\|_{p}=1$ such that
\begin{equation}\label{eq:opt}
    \liminf_{j\to\infty}\inf_{\Psi}\|(u_j)_\Psi-1\|_{p}=0\qquad \text{and} \qquad \limsup_{j\to\infty}\frac{\|u_j\|^{p}_{p}-\|Qu_j\|^{p}_{q}}{\inf_{\Psi}\|(u_j)_\Psi-1\|_{r}^r}<\infty\,.
\end{equation}
This will be shown in Subsections \ref{sec:optimal1} (for $r=2$) and \ref{sec:optimal2} (for $r=p$).

Turning to the dual HWY-inequality \eqref{eq:HWYspheredual}, the result in Proposition \ref{prop:localstabilitydual}, and thus in Theorems \ref{thm:HWYstabilitydualsphere} and \ref{thm:HWYstabilitydual}, is optimal in the sense that the exponent $2$ of the $L^{q'}$-distance cannot be decreased. To be more specific, we can find a sequence $(v_j)\subset L^{{q'}}(B^d)$ with $\|v_j\|_{q'}=1$ such that \begin{equation}\label{eq:optdual}
    \liminf_{j\to\infty}\inf_{\Phi}\|[v_j]_\Phi-1\|_{q'}=0\qquad \text{and} \qquad \limsup_{j\to\infty}\frac{\|v_j\|^{q'}_{q'}-\|Sv_j\|^{q'}_{p'}}{\inf_{\Phi}\|[v_j]_\Phi-1\|_{q'}^{2}}<\infty\,.
\end{equation}
This will be shown in Subsection \ref{sec:optimal3}.


\subsection{Optimality of the quadratic power for the $L^2(\mathbb S^{d-1})$-distance}\label{sec:optimal1}

 To prove \eqref{eq:opt} in case $r=2$, we fix a function $0\not\equiv \phi\in C^\infty(\mathbb S^{d-1})\subseteq L^p(\mathbb S^{d-1})$ with $\int_{\mathbb S^{d-1}}\phi\,\mathrm d\mu=\int_{\mathbb S^{d-1}}\phi\omega_i\,\mathrm d\mu(\omega)=0$, $i=1,\ldots,d$, and pick $(u_j)$ to be a subsequence of
$$u_\varepsilon\coloneqq \lambda_\varepsilon(1+\varepsilon\phi)$$ with $\varepsilon\to0$ and $\lambda_\varepsilon>0$ chosen such that $\|u_\varepsilon\|_{p}=1$. Note that $\lambda_\varepsilon=1+ o(1)$. As $2<p<q$, we can expand the HWY-deficit to second order, which gives
$$\|u_\varepsilon\|^{p}_{p}-\|Qu_\varepsilon\|^{p}_{q}=\varepsilon^2 \frac{p(p-1)}{2}\left(\|\phi\|_2^2-\frac{q-1}{p-1}\|Q\phi\|_{2}^2\right)+o(\varepsilon^2)\,.$$
This is clearly bounded from above by a constant times $\varepsilon^2\|\phi\|_2^2$. Thus, we can conclude \eqref{eq:opt} in case $r=2$, once
\begin{equation*}
    \inf_{\Psi}\|(u_\varepsilon)_\Psi-1\|^2_{2}=\varepsilon^2\|\phi\|_2^2+o(\varepsilon^2)
\end{equation*}
is verified. The latter follows \textit{mutatis mutandis} from the proof of \cite[Eq.~(5.2)]{Frank2024B}. Indeed, the $L^2$- and $L^{p}$-norms replace the $W^{1,2}$- and $W^{1,4}$-norms, and instead of the conformal bounds for $W^{1,4}$, we can directly use conformal invariance of the $L^{p}$-norm.


\subsection{Optimality of the power $p$ for the $L^p(\mathbb S^{d-1})$-distance}\label{sec:optimal2}
To prove \eqref{eq:opt} in case $r=p$, we consider $(u_j)$ as a subsequence of the two-parameter family of functions $$u_{\delta,\eta}\coloneqq \lambda_{\delta,\eta}(1+\delta(1)_{\Psi_\eta})$$ with $\eta\to \eta_*\in \mathbb S^{d-1}$ and $\delta\to0$. Here $\lambda_{\delta,\eta}$ is chosen such that $\|u_{\delta,\eta}\|_{p}=1$, and $\Psi_\eta$ denotes the M\"obius transformation associated with $\eta\in B^d$ from \eqref{eq:psieta}.

As $\eta\to \eta_*\in \mathbb S^{d-1}$, $(1)_{\Psi_\eta}(\omega)\to 0$ and $Q(1)_{\Psi_\eta}(\omega)=(1)^{p/q}_{\Psi_\eta}(\omega)\to 0$ for all $\omega\in \mathbb S^{d-1}\setminus\{\eta_*\}$. Applying the Brezis--Lieb lemma and conformal invariance, we obtain the expansions
\begin{align}\label{eq:expa1}
    \|1+\delta(1)_{\Psi_\eta}\|^{p}_{p}&=\|1\|_{p}^p\left(1+\delta^{p}\right)+o_{|\eta|\to 1}(1)\,,\\
    \|Q(1+\delta(1)_{\Psi_\eta})\|^{q}_{q}&=\|Q1\|^{q}_{q}\left(1+\delta^{q}\right)+o_{|\eta|\to 1}(1)\,,\notag
\end{align}
uniformly in $\delta\in (0,1)$. As a consequence of \eqref{eq:expa1}, we have
\begin{equation*}\lambda_{\eta,\delta}=\left(1+\delta^{p}\right)^{-\frac{1}p}+o_{|\eta|\to 1}(1)\end{equation*}
for fixed $\delta$. We now choose $\eta$ as a function of $\delta$ (but for simplicity we do not reflect this in the notation). Choosing $\eta\to\eta_*$ sufficiently faster than $\delta\to 0$, we can ensure that the error in $o_{|\eta|\to1}(1)$ is controlled by $\delta^{p+\epsilon}$ for some fixed $\epsilon>0$. This implies 
\begin{equation}
\label{eq:lambda}\lambda_{\eta,\delta}=1-\frac1p \delta^p (1+o(1)) 
\end{equation} and 
\begin{align*}
    \|u_{\delta,\eta}\|^{p}_{p}- \|Qu_{\delta,\eta}\|^{p}_{q}&=|\lambda_{\eta,\delta}|^p\left(1+\delta^{p}-\left(1+\delta^{q}\right)^{\frac{d-1}{d}}+o_{|\eta|\to 1}(1)\right)\\&=\delta^{p}(1+o(1))\,.
\end{align*} 

Therefore, if we can prove
\begin{equation}\inf_\Psi\|(u_{\delta,\eta})_\Psi-1\|^{p}_{p}=\delta^{p}\|1\|^{p}_{p}+o(\delta^{p})\,,\label{eq:expdist}  
\end{equation}
we can conclude \eqref{eq:opt} for $r=p$. Note that we can dismiss the prefactor $\lambda_{\delta,\eta}$ in the definition of $u_{\delta,\eta}$ due to \eqref{eq:lambda}. If we take the identity as competitor in the infimum and use conformal invariance of the $L^{p}$-norm, we obtain the upper bound
$$
\inf_\Psi\|(1+\delta(1)_{\Psi_\eta})_\Psi-1\|_{p}\leq \delta \|1\|_{p}\,.
$$  
Let $\eta'\in B^d$ (depending on $\delta$ and $\eta$) be such that
$$
\|(1+\delta(1)_{\Psi_\eta})_{\Psi_{\eta'}}-1\|^{p}_{p}
\leq \inf_\Psi\|(1+\delta(1)_{\Psi_\eta})_{\Psi}-1\|^{p}_{p} + o(\delta^p) \,.
$$
(Here we use the same argument as below \eqref{eq:infimummobius} to reduce the infimum over $\Psi$ to an infimum over elements of $B^d$.) Thus, by the triangle inequality, we obtain the lower bound 
\begin{equation*}
\inf_\Psi\|(1+\delta(1)_{\Psi_\eta})_{\Psi}-1\|_{p}\geq \|(1)_{\Psi_{\eta'}}-1\|_{p}-\|\delta (1)_{\Psi_{\eta}\circ\Psi_{\eta'}}\|_{p}+o(\delta^{p})\,.
\end{equation*}
Again as a consequence of conformal invariance, by comparison with the upper bound we deduce 
\begin{equation}\label{eq:deltabound}
   \|(1)_{\Psi_{\eta'}}-1\|^{p}_{p}\lesssim \delta^{p}\,.
\end{equation}

This implies $\eta'\to 0$. (Indeed, otherwise it would have a limit point in $\overline{B^d}\setminus\{0\}$, and the Brezis--Lieb lemma would imply that along the corresponding subsequence the left side of \eqref{eq:deltabound} remained positive, contradicting the bound.) Since $\eta'\to 0$, we can use the uniform expansion 
\begin{equation}
    \label{eq:pointwisediff}
    (1)_{\Psi_{\eta'}}-1=\operatorname{const}\eta'(1+o(1))\,.
\end{equation}

Together with \eqref{eq:deltabound}, this gives $|\eta'|\lesssim \delta$. 
Meanwhile, an application of Young's elementary inequality tells us that
\begin{equation}
   \|(1+\delta(1)_{\Psi_\eta})_{\Psi_{\eta'}}-1\|^{p}_{p}\geq \|\delta (1)_{\Psi_{\eta}\circ \Psi_{\eta'}}\|^{p}_{p}-p\int_{\mathbb S^{d-1}}|(1)_{\Psi_{\eta'}}-1|^{p-1}\delta |(1)_{\Psi_{\eta}\circ\Psi_{\eta'}}|\,\mathrm d\mu\,.
\label{eq:Young1}\end{equation}

We claim that
\begin{equation}
    \label{eq:counterexampleproof}
    \|(1)_{\Psi_{\eta}\circ\Psi_{\eta'}}\|_1\to 0\,.
\end{equation}
Indeed, by Corollary \ref{cor:mobiusclass}, there are $\eta''\in B^d$ and $A\in O(d)$ such that $\Psi_{\eta}\circ\Psi_{\eta'}=A\circ\Psi_{\eta''}$. If $\eta''$ stayed away from the boundary of $B^d$, it would have a subsequence that converges to a limit in $B^d$. By compactness of $O(d)$ and the fact that $\eta'$ stays away from the boundary of $B^d$, a subsequence of $\Psi_\eta= A\circ\Psi_{\eta''}\circ\Psi_{\eta'}^{-1}$ would converge to a M\"obius transformation of $\Sph^{d-1}$, contradicting $|\eta|\to 1$. This shows that $|\eta''|\to 1$. Since $1<p$ it follows as in \cite[Lemma~7]{Frank2024B} that $\|(1)_{\Psi_{\eta}\circ\Psi_{\eta'}}\|_1 = \|(1)_{\Psi_{\eta''}}\|_1\to 0$, as claimed in \eqref{eq:counterexampleproof}.

Combining the pointwise bound \eqref{eq:pointwisediff}, the estimate $|\eta'|\lesssim\delta$, and \eqref{eq:counterexampleproof} yields
$$
\int_{\mathbb S^{d-1}}|(1)_{\Psi_{\eta'}}-1|^{p-1}\delta |(1)_{\Psi_{\eta''}}|\,\mathrm d\mu\lesssim \delta^{p}\|(1)_{\Psi_{\eta''}}\|_1= o(\delta^{p})\,.
$$
In conclusion, \eqref{eq:Young1} yields
$$
\|(1+\delta(1)_{\Psi_\eta})_{\Psi_{\eta'}}-1\|^{p}_{p}\geq \delta^p \|1\|_p^p + o(\delta^p) \,.
$$
This implies the lower bound in \eqref{eq:expdist} and thereby completes the proof.


\subsection{Optimality of the quadratic power for the $L^{q'}(B^d)$-distance}
\label{sec:optimal3}

To prove \eqref{eq:optdual}, we consider $(v_j)$ as a subsequence of the two-parameter family of functions $$v_{\delta,\eta}\coloneqq \lambda_{\delta,\eta}(1+\delta[1]_{\Phi_\eta})$$ with $\eta\to \eta_*\in \mathbb S^{d-1}$, $\delta\to0$, and $\lambda_{\delta,\eta}$ chosen such that $\|v_{\delta,\eta}\|_{q'}=1$. As in the previous subsection, by applying the Brezis--Lieb Lemma and conformal invariance, we expand
\begin{equation*}
    \|v_{\delta,\eta}\|^{{q'}}_{q'}- \|Sv_{\delta,\eta}\|^{q'}_{p'}= 1+\delta^{{q'}}-\left(1+\delta^{{p'}}\right)^{\frac{q'}{p'}}+o_{|\eta|\to 1}(1)
    \,
\end{equation*}
and 
$$\lambda_{\eta,\delta}=\left(1+\delta^{q'}\right)^{-\frac{1}{q'}}+o_{|\eta|\to 1}(1)\,$$
uniformly in $\delta\in (0,1)$. 
If we choose $\eta\to\eta_*$ sufficiently faster than $\delta\to 0$, we obtain
$$ \|v_{\delta,\eta}\|^{{q'}}_{q'}- \|Sv_{\delta,\eta}\|^{{q'}}_{p'}=\delta^{q'}(1+o_{\delta\to 0}(1))$$
and $$\lambda_{\eta,\delta}=1-\frac 1{q'}\delta^{q'}(1+o_{\delta\to 0}(1))\,.$$
Once again, we are left to determine the asymptotic behavior of the distance. It suffices to prove
\begin{equation*}\inf_\Phi\|[v_{\delta,\eta}]_\Phi-1\|^{q'}_{q'}=\delta^{q'}\|1\|^{q'}_{q'}+o_{\delta\to0}(\delta^{q'})
\end{equation*}
to conclude \eqref{eq:optdual}. 
This follows identically as the proof of \eqref{eq:expdist} with exponent $q'$ instead of $p$; we omit the details.


\appendix


\section{Conformal invariance}\label{sec:confinv}

Our goal in this appendix is to prove the invariance under M\"obius transformations of the HWY-inequality and its dual. For background on M\"obius transformations, we refer to \cite{Ratcliffebook}.
\subsubsection*{The inequality on the upper half-space}

By definition, a M\"obius transformation of $\R^d\cup\{\infty\}$ is a finite composition of reflections in generalized spheres. (A generalized sphere is either a sphere or an affine hyperplane.) Also by definition, a M\"obius transformation of $\R^d_+$ is a M\"obius transformation of $\R^d\cup\{\infty\}$ that leaves $\R^d_+$ invariant.

It is easy to see \cite[Theorem 4.4.1]{Ratcliffebook} that by restriction to the boundary $\partial\R^d_+$, identified with $\R^{d-1}$, a M\"obius transformation of $\R^d_+$ gives rise to a M\"obius transformation of $\R^{d-1}\cup\{\infty\}$. Conversely, any M\"obius transformation of $\R^{d-1}\cup\{\infty\}$ has a unique extension to a M\"obius transformation of $\R^d\cup\{\infty\}$, the so-called Poincar\'e extension \cite[Section 4.4]{Ratcliffebook}. It is defined as follows: If $\psi$ is a reflection in a plane $\{ x' \in\R^{d-1} :\ a'\cdot x' = t\}$  (respectively in a sphere $\{ x'\in \R^{d-1}:\ |x'-a'| = r\}$), then $\tilde\psi$ is defined to be the reflection in the plane $\{ x\in\R^d :\ a'\cdot x' = t\}$ (respectively in the sphere $\{ x\in\R^d :\ | x- (a',0)| = r\}$), where $x=(x',x_d)$. Clearly, in both cases $\tilde\psi$ is an extension of $\psi$ in the sense that $\tilde\psi(x',0) = (\psi(x'),0)$ for $x'\in\R^{d-1}$. Moreover, in both cases $\tilde\psi$ leaves $\R^d_+$ invariant. For a general M\"obius transformation $\psi$ of $\R^{d-1}\cup\{\infty\}$, the Poincar\'e extension $\tilde\psi$ is defined by writing $\psi$ as a composition of reflections in generalized spheres and extending each reflection. The fact that this definition is independent of the chosen representation as composition follows from the fact that if $\tilde\psi_1$ and $\tilde\psi_2$ are two such extensions of $\psi$, then $\tilde\psi_1 \circ \tilde\psi_2^{-1}$ is a M\"obius transformation of $\R^d\cup\{\infty\}$ that fixes each point of $\R^{d-1}\times\{0\}$ and is therefore the identity by a well-known theorem as given in \cite[Theorem 4.3.6]{Ratcliffebook}, for instance.

The Poincar\'e extension is relevant in our context because of the following formula.

\begin{lemma}\label{lem:poincare}
	For any $f\in L^\frac{2(d-1)}{d-2}(\R^{d-1})$ and any M\"obius transformation $\psi$ of $\R^{d-1}\cup\{\infty\}$,
	$$
	P\left( J_\psi^\frac{d-2}{2(d-1)} \, f\circ\psi\right) = J_{\tilde\psi}^\frac{d-2}{2d}\, (Pf)\circ\tilde\psi \,.
	$$	
\end{lemma}

\begin{proof}
	The Poincar\'e extension of $\psi(\xi)=\xi-\xi_0$ (with $\xi_0\in\R^{d-1}$) is $\tilde\psi(x)=(x'-\xi_0,x_d)$, and the Poincar\'e extension of $\psi(\xi)=b\xi$ (with $b\in\R_+$) is $\tilde\psi(x)=bx$. For both transformations the assertion of the proposition is clear. Momentarily, we will show the assertion for $\psi(\xi) = \xi/|\xi|^2$, for which we have $\tilde\psi(x)=x/|x|^2$. The assertion for general $\psi$ then follows from the fact that the Poincar\'e extension of a composition is the composition of the Poincar\'e extensions, which is a consequence of the uniqueness of the Poincar\'e extension.

    It remains to prove the formula for the inversion; see also \cite[Proof of Theorem 4.1]{Hang2008}. For given $f\in L^\frac{2(d-1)}{d-2}(\R^{d-1})$ let $\tilde f(\xi)\coloneqq  |\xi|^{-(d-2)} f(\xi/|\xi|^2)$. We compute 
    \begin{align*}
        (P\tilde f)(x) & = 
         \frac{2}{|\Sph^{d-1}|} \int_{\R^{d-1}} \frac{x_d}{(|x'-|\tilde\xi|^{-2} \tilde\xi|^2 + x_d^2)^{d/2}} \frac{1}{|\tilde\xi|^{d}} f(\tilde\xi)\,\mathrm d\tilde\xi \\
        & = |x|^{- d+2} \frac{2}{|\Sph^{d-1}|} \int_{\R^{d-1}} \frac{|x|^{-2} x_d}{(||x|^{-2} x'- \tilde\xi|^2 + (|x|^{-2} x_d)^2)^{d/2}} f(\tilde\xi)\,\mathrm d\tilde\xi = |x|^{-d+2} \, (Pf)(x/|x|^2) \,.
    \end{align*} 
    This is the claimed formula.   
\end{proof}

\subsubsection*{Passing from the half-space to the ball}

Let $\mathcal S:\R^{d-1}\cup\{\infty\}\to\Sph^{d-1}$ be the inverse stereographic projection, given by
$$
\mathcal S(\xi) \coloneqq  \left( \frac{2\xi}{1+|\xi|^2}, \frac{1-|\xi|^2}{1+|\xi|^2} \right),
\qquad \xi\in\R^{d-1} \,,
$$
and recall that $J_\mathcal S(\xi) = (2/(1+|\xi|^2))^{d-1}$ is its Jacobian. Moreover, let $\Sigma:\R^d_+\to B^d$ be given by
$$
\Sigma(x) \coloneqq  \left( \frac{2x'}{(1+x_d)^2 + |x'|^2}, \frac{1-|x|^2}{(1+x_d)^2+|x'|^2} \right),
\qquad x\in\R^d_+ \,.
$$
Note that $\Sigma$, or rather its extension to the closure of $\R^d_+$, 
coincides with $\mathcal S$ on $\partial\R^d_+$, identified with $\R^{d-1}$.

Then a similar computation as in the proof of Lemma \ref{lem:poincare} shows that
$$
Q\left(J_{\mathcal S^{-1}}^\frac{d-2}{2(d-1)} \, f\circ \mathcal S^{-1}\right) = J_{\Sigma^{-1}}^\frac{d-2}{2d} \, (Pf) \circ\Sigma^{-1} \,.
$$
Thus, if a function $f$ on $\R^{d-1}$ and a function $u$ on $\Sph^{d-1}$ are related by
$$
    u = |\Sph^{d-1}|^\frac{d-2}{2(d-1)} \, J_{\mathcal S^{-1}}^\frac{d-2}{2(d-1)} \, f\circ \mathcal S^{-1} \,,
$$
then, recalling that we use probability measures on $\Sph^{d-1}$ and $B^d$,
\begin{equation}
    \label{eq:confinv}
    \| u \|_{L^\frac{2(d-1)}{d-2}(\mathbb S^{d-1})} = \|f\|_{L^\frac{2(d-1)}{d-2}(\R^{d-1})}
    \text{ and}\quad
    \|Qu \|_{L^\frac{2d}{d-2}(B^d)} = d^\frac{d-2}{2d} |\Sph^{d-1}|^{\frac{d-2}{2d(d-1)}} \|Pf\|_{L^\frac{2d}{d-2}(\R^d_+)} \,.
\end{equation}

\subsubsection*{The inequality on the ball}

By definition, a M\"obius transformation of $\Sph^{d-1}$ is a map $\Psi:\Sph^{d-1}\to\Sph^{d-1}$ such that $\mathcal S^{-1} \circ\Psi\circ\mathcal S$ is a M\"obius transformation of $\R^{d-1}\cup\{\infty\}$. Also by definition, a M\"obius transformation of $B^d$ is a M\"obius transformation of $\R^d\cup\{\infty\}$ that leaves $B^d$ invariant.

The Poincar\'e extension of a M\"obius transformation $\Psi$ of $\Sph^{d-1}$ is, by definition, the M\"obius transformation $\tilde\Psi\coloneqq \Sigma\circ\tilde\psi\circ\Sigma^{-1}$ of $B^d$ where $\psi \coloneqq \mathcal S^{-1}\circ\Psi\circ\mathcal S$. Just like for the half-space, this defines a bijection between M\"obius transformations on $\Sph^{d-1}$ and on $B^d$. 

In view of these definitions, we can state Lemma \ref{lem:poincare} equivalently as follows.

\begin{lemma}
    \label{lem:poincareball}
	For any $u\in L^\frac{2(d-1)}{d-2}(\Sph^{d-1})$ and any M\"obius transformation $\Psi$ of $\Sph^{d-1}$,
	$$
	Q( J_\Psi^\frac{d-2}{2(d-1)} \, u\circ\Psi) = J_{\tilde\Psi}^\frac{d-2}{2d}\, (Qu)\circ\tilde\Psi \,.
	$$	
\end{lemma}

We note that by taking $u=1$, so that $Qu=1$, we find the identity
\begin{equation*}
    Q J_\Psi^\frac{d-2}{2(d-1)} = J_{\tilde\Psi}^\frac{d-2}{2d} \,.
\end{equation*}

\subsubsection*{M\"obius transformations of $B^d$}

Sometimes it is convenient to have an explicit form of M\"obius parametrizations. This is well known to experts, but we include the details for the sake of completeness. We let, for $\eta\in B^d$,
\begin{equation}
    \label{eq:phieta}
    \Phi_\eta(y)\coloneqq  \frac{(1-|\eta|^2)(y-\eta) - |y-\eta|^2\eta}{1-2\eta\cdot y + |\eta|^2|y|^2}     \qquad\text{for all}\  y\in B^d \,.
\end{equation}
As usual, $O(d)$ denotes the set of orthogonal transformations of $\R^d$.

\begin{lemma}
	M\"obius transformations of $B^d$ are precisely those maps $\Phi$ given by
	$$
	\Phi(y) = A \Phi_\eta(y)\,,
	\qquad y\in B^d \,,
	$$
	with $A\in O(d)$ and $\eta\in B^d$.
\end{lemma}

\begin{proof}
	Denoting $y^*\coloneqq y/|y|^2$, we find
	\begin{equation}
	    \label{eq:phieta2}
        	\Phi_\eta(y) = -\eta + (1-|\eta|^2) (y^*-\eta)^* \,.
	\end{equation}
    Thus, $\Phi_\eta$ is a M\"obius transformation of $\R^d\cup\{\infty\}$. A tedious but elementary computation shows that $|\Phi_\eta(y)|<1$ when $|y|<1$, so $\Phi_\eta$ leaves $B^d$ invariant. Hence, $\Phi_\eta$ is a M\"obius transformation of $B^d$, and this remains true after composing it with an orthogonal transformation.
		
	Conversely, assume that $\Phi$ is a M\"obius transformation of $B^d$, and set $\eta= \Phi^{-1}(0)$. Then $\Phi_\eta \circ\Phi^{-1}$ is a M\"obius transformation of $B^d$ that fixes the point $0$ and, therefore, by \cite[Theorem 4.4.8]{Ratcliffebook}, it is an orthogonal transformation. It is instructive to review the proof in some more detail: With $J(y) \coloneqq  y^* = y/|y|^2$, we consider $\phi \coloneqq  J\circ \Phi_\eta \circ\Phi^{-1}\circ J$. This is a M\"obius transformation of $\R^d\cup\{\infty\}$ that fixes $\infty$. According to \cite[Theorem 4.3.2]{Ratcliffebook}, $\phi$ is a Euclidean similarity, that is, there are $A\in O(d)$, $k>0$, and $z\in\R^d$ such that $\phi(x) = kAx + z$. Each one of the four maps making up $\phi$ leaves $\Sph^{d-1}$ invariant. Meanwhile, $x\mapsto kAx+z$ maps $\Sph^{d-1}$ to the sphere of radius $k$ centered at $z$. It follows that $k=1$ and $z=0$. Thus,
	$$
	\Phi_\eta\circ\Phi^{-1}(y) = J(\phi(Jy)) = J(A(Jy)) = \frac{AJy}{|AJy|^2} = Ay \,,
	\qquad y\in B^d \,,
	$$
	as claimed.
\end{proof}

A tedious but straightforward computation shows that the Jacobian of $\Phi_\eta$ is given by
\begin{equation}
    \label{eq:phietajac}
    (J_{\Phi_\eta}(y))^\frac{1}{d} = \frac{1-|\eta|^2}{1- 2\eta\cdot y + |\eta|^2|y|^2}
    \qquad\text{for all}\ y\in B^d \,.
\end{equation}
(The expression \eqref{eq:phieta2} is useful when performing this computation.)

\subsubsection*{M\"obius transformations of $\Sph^{d-1}$}
For $\eta\in B^d$, let
\begin{equation}
    \label{eq:psieta}
    \Psi_\eta(\omega)\coloneqq  \frac{(1-|\eta|^2)(\omega-\eta) - (1-2\eta\cdot\omega+|\eta|^2) \eta}{1-2\eta\cdot\omega + |\eta|^2} \qquad\text{for all}\ \omega\in \mathbb S^{d-1}\,.
\end{equation}

\begin{corollary}\label{cor:mobiusclass}
	M\"obius transformations of $\Sph^{d-1}$ are precisely those maps $\Psi$ given by
	$$
	\Psi(\omega) = A \Psi_\eta(\omega)\,,
	\qquad \omega\in\Sph^{d-1} \,,
	$$
	with $A\in O(d)$ and $\eta\in B^d$.
\end{corollary}

\begin{proof}
    This follows from the fact that $\Phi_\eta$ is the Poincar\'e extension of $\Psi_\eta$ and that the Poincar\'e extension provides a bijection between M\"obius transformations of $\Sph^{d-1}$ and $B^d$.
\end{proof}

The Jacobian of $\Psi_\eta$ is given by
\begin{equation*}
    (J_{\Psi_\eta}(\omega))^\frac{1}{d-1}=\frac{1-|\eta|^2}{1-2\eta\cdot \omega+|\eta|^2}
    \qquad\text{for all}\ \omega\in\Sph^{d-1} \,.
\end{equation*}
Indeed, this follows, for example, from \eqref{eq:phietajac} together with the formula
\begin{equation}
    \label{eq:poincarejac}
    J_\Psi(\omega)^\frac{1}{d-1} = J_{\tilde\Psi}(\omega)^\frac1d\,,
    \qquad \omega\in\Sph^{d-1}\,,
\end{equation}
relating the Jacobian of a M\"obius transformation $\Psi$ of $\Sph^{d-1}$ and its Poincar\'e extension. 

To prove the latter, we recall that for any M\"obius transformation $\psi$ of $\R^d\cup\{\infty\}$ we have
\begin{equation}
    \label{eq:mobiusnorm1}
    |\psi(x) - \psi(x')|^2 = J_\psi(x)^\frac1d \, |x-x'|^2 \, J_\psi(x')^\frac1d \,,
    \qquad x,x'\in\R^d \,.
\end{equation}
Applying this formula in $\R^{d-1}$ and using a similar formula for the inverse stereographic projection, we find for any M\"obius transformation $\Psi$ of $\Sph^{d-1}$ that
\begin{equation}
    \label{eq:mobiusnorm2}
    |\Psi(\omega) - \Psi(\omega')|^2 = J_\Psi(\omega)^\frac1{d-1} \, |\omega-\omega'|^2 \, J_\Psi(\omega')^\frac1{d-1} \,,
    \qquad \omega,\omega'\in\Sph^{d-1} \,.
\end{equation}
Taking $\psi=\tilde\Psi$, $x=\omega$, and $x'=\omega'$ in \eqref{eq:mobiusnorm1} and comparing with \eqref{eq:mobiusnorm2}, we arrive at \eqref{eq:poincarejac}.


\section{A unified notion of distance} 

In this appendix we propose a way of measuring the distance to the set of optimizers that captures a quadratic behavior close-by and a $p$-homogeneous behavior far away.

For $1\leq r<\infty$ and functions $f,f^*\in L^r(X)$ on a measure space $(X, \mathcal A, \sigma)$, we set
$$
\Pi_{r}(f,f^*) \coloneqq \frac{1}{\|f^*\|_{L^r(X)}^r}\left(\| f-f^*\|^r_{L^{r}(\{|f-f^*|> |f^*|\})}+ \left\| (f-f^*)|f^*|^{r/2-1}\right\|^2_{L^{2}(\{|f-f^*|\leq |f^*|\})}\right).
$$
Clearly, we have
$$
    \|f^*\|_{L^r(X)}^r \, \Pi_{r}(f,f^*) \simeq 
\begin{cases}
\int_X \max\{|f^*|^{r/2-1}(f-f^*)^2,|f-f^*|^r\}\, \mathrm d\sigma & \text{if}\ r\geq 2 \,,\\
\int_X \min\{|f^*|^{r/2-1}(f-f^*)^2,|f-f^*|^r\}\, \mathrm d\sigma &
\text{if}\ r \leq 2 \,,
\end{cases}
$$
where $\simeq$ mean $\lesssim$ and $\gtrsim$ with constants depending only on $r$. When $r\geq 2$, it is easy to see that
$$
\|f^*\|_{L^r(X)}^r \, \Pi_{r}(f,f^*) \simeq \| f-f^*\|^r_{L^{r}(X)}+ \| f-f^*\|^2_{L^{2}(X,\,|f^*|^{r-2} d\sigma)} \,,
$$
while for $r<2$ the argument in \eqref{eq:mintoq'} shows that
$$
\|f^*\|_{L^r(X)}^r \, \Pi_{r}(f,f^*) \gtrsim \|f-f^*\|^2_{L^{r}(X)}\|f^*\|_{L^r(X)}^{r-2}
\qquad\text{if}\ \|f-f^*\|_{L^{r}(X)} \leq \|f^*\|_{L^r(X)} \,.
$$

Here is a variant of Theorems \ref{thm:HWYstabilitysphere} and \ref{thm:HWYstabilitydualsphere} in terms of this quantity.

\begin{corollary}\label{cor:HWY}
    Let $d\geq 3$. There is a $c_d>0$ such that for all $0\neq u\in L^p(\mathbb S^{d-1})$ and $0\neq v\in L^{q'}(B^d)$ we have
\begin{align*}
    1 -\frac{\|Qu\|_q}{ \|u\|_p} \geq c_d\inf_{u^*}\Pi_p(u,u^*)
\qquad\text{and}\qquad
    1 -\frac{\|Sv\|_{p'}}{ \|v\|_{q'}} \geq c_d \inf_{v^*}\Pi_{q'}(v,v^*)\,,
    \end{align*}
    where the infima are taken over all optimizers of 
    \eqref{eq:HWYsphere} 
     and \eqref{eq:HWYspheredual}, respectively.
\end{corollary}

A similar stability result holds for the HWY-inequality for the upper half-space and its dual, for the $\sigma_2$-curvature inequality on the sphere \cite{Frank2024B}, and for the $p$-Sobolev inequality on Euclidean space \cite{Figalli2022}.

\begin{proof}[Proof of Corollary \ref{cor:HWY}]
	While the proof for the HWY-inequality is direct (just up to repositioning $\lambda$), for the dual version we have to make more changes to the proof as it uses a new notion of distance. By contradiction, assume that there is a sequence $(v_n)\subset L^{q'}(B^d)$ with
	\begin{equation}
		\label{eq:thmproofass3}
		\frac{ 1-\|Sv_n\|_{p'}/\| v_n \|_{q'}}{\inf_{\lambda,\Phi}   \int_{B^d} \min\{|v_n-\lambda\,[1]_{\Phi}|^{q'},|v_n-\lambda\,[1]_{\Phi}|^{2}\}\,\mathrm d\nu } \to 0
	\end{equation} as $n\to\infty$. We may normalize the sequence $(v_n)$ by $\| v_n \|_{q'} = 1$.    Since
	\begin{equation*}
	    \inf_{\lambda,\Phi}  \int_{B^d} \min\{|v_n-\lambda\,[1]_{\Phi}|^{q'},|v_n-\lambda\,[1]_{\Phi}|^{2}\}\,\mathrm d\nu \leq \|v_n \|_{q'}^{q'}=1\,,
	\end{equation*} 
	we deduce from \eqref{eq:thmproofass3} that $ \|Sv_n\|_{p'}\to 1$ as $n\to\infty$. Hence, by Proposition \ref{prop:qualdual}, $\inf_{\Phi,\lambda\in\{\pm 1\}} \| v_n - \lambda\,[1]_\Phi\|_{q'}\to 0$ as $n\to\infty$. An application of Proposition \ref{prop:localstabilitydual} in the form given by \eqref{eq:stabbdd} then gives a positive lower bound that contradicts \eqref{eq:thmproofass3}.
\end{proof}


\subsection*{Acknowledgements}
    Partial support through the US National Science Foundation grant DMS-1954995 (R.L.F.) and the German Research Foundation grants EXC-2111-390814868 (R.L.F.), FR 2664/3-1 (R.L.F. \& J.W.P.), and TRR 352-Project-ID 470903074 (R.L.F., J.W.P., \& L.R.) is acknowledged.

\begin{thebibliography}{AKRW24}

\bibitem[AKN23]{allen_2023}
M.~Allen, D.~Kriventsov, and R.~Neumayer, \emph{Sharp quantitative
  {F}aber-{K}rahn inequalities and the {A}lt-{C}affarelli-{F}riedman
  monotonicity formula}, Ars Inven. Anal. \textbf{2023} (2023), no.~1, 1--49.

\bibitem[AKRW24]{andrade2024}
J.~H. Andrade, T.~König, J.~Ratzkin, and J.~Wei, \emph{Quantitative stability
  of the total {$Q$}-curvature near minimizing metrics}, 2024.
  
\bibitem[Bec92]{beckner_sobolev_1992}
W.~Beckner, \emph{Sobolev inequalities, the {P}oisson semigroup, and analysis
  on the sphere ${S^n}$.}, Proc. Natl. Acad. Sci. USA \textbf{89} (1992),
  no.~11, 4816--4819.

\bibitem[BE91]{Bianchi1991}
G.~Bianchi and H.~Egnell, \emph{A note on the {S}obolev inequality}, J. Funct.
  Anal. \textbf{100} (1991), no.~1, 18--24.

\bibitem[BDNS23]{Bonforte2023}
M.~Bonforte, J.~Dolbeault, B.~Nazaret, and N.~Simonov, \emph{Constructive
  stability results in interpolation inequalities and explicit improvements of
  decay rates of fast diffusion equations}, Discrete Contin. Dyn. Syst.
  \textbf{43} (2023), no.~3\&4, 1070--1089.

\bibitem[BPV15]{brasco_2015}
L.~Brasco, G.~De Philippis, and B.~Velichkov, \emph{Faber-{K}rahn inequalities
  in sharp quantitative form}, Duke Math. J. \textbf{164} (2015), no.~9,
  1777--1831.

\bibitem[Bre11]{brezisbook}
H.~Brezis, \emph{Functional analysis, {S}obolev spaces and partial differential
  equations}, Springer, New York, 2011.

\bibitem[BL85]{BREZIS198573}
H.~Brezis and E.~H. Lieb, \emph{Sobolev inequalities with remainder terms}, J.
  Funct. Anal. \textbf{62} (1985), no.~1, 73--86.


\bibitem[BDS24]{brigati2024}
G.~Brigati, J.~Dolbeault, and N.~Simonov, \emph{Logarithmic {S}obolev and
  interpolation inequalities on the sphere: Constructive stability results},
  Ann. Inst. H. Poincar{\'e} C Anal. Non Lin{\'e}aire \textbf{41} (2024),
  no.~5, 1289--1321.


\bibitem[Car21]{Carleman1921}
T.~Carleman, \emph{Zur {T}heorie der {M}inimalflächen}, Math. Z. \textbf{9}
  (1921), 154--160.

\bibitem[Car17]{Carlen2017}
E.~A. Carlen, \emph{Duality and stability for functional inequalities}, Ann.
  Fac. Sci. Toulouse Math. (6) \textbf{26} (2017), no.~2, 319--350.

\bibitem[Car25]{carlen_2025}
\bysame, \emph{Stability for the logarithmic {H}ardy-{L}ittlewood-{S}obolev
  inequality with application to the {K}eller-{S}egel equation}, J. Funct.
  Anal. \textbf{288} (2025), no.~6, 110818.

\bibitem[CFL14]{Carlen2014}
E.~A. Carlen, R.~L. Frank, and E.~H. Lieb, \emph{{S}tability estimates for the
  lowest eigenvalue of a {S}chrödinger operator}, Geom. Funct. Anal.
  \textbf{24} (2014), no.~1, 63--84.

\bibitem[CFMP09]{Cianchi2009}
A.~Cianchi, N.~Fusco, F.~Maggi, and A.~Pratelli, \emph{The sharp {S}obolev
  inequality in quantitative form}, J. Eur. Math. Soc. \textbf{11} (2009),
  no.~5, 1105--1139.

\bibitem[CL12]{cicalese_2012}
M.~Cicalese and G.~P. Leonardi, \emph{A selection principle for the sharp
  quantitative isoperimetric inequality}, Arch. Ration. Mech. Anal.
  \textbf{206} (2012), no.~2, 617--643.

\bibitem[CFW13]{Chen2013}
S.~Chen, R.~L. Frank, and T.~Weth, \emph{Remainder terms in the fractional
  {S}obolev inequality}, Indiana Univ. Math. J. \textbf{62} (2013), no.~4,
  1381--1397.

\bibitem[CLT24]{chen_2024}
L.~Chen, G.~Lu, and H.~Tang, \emph{Stability of {H}ardy-{L}ittlewood-{S}obolev
  inequalities with explicit lower bounds}, Adv. Math. \textbf{450} (2024),
  109778.

\bibitem[CLT25a]{chen_2025}
\bysame, \emph{Optimal asymptotic lower bound for stability of fractional
  {S}obolev inequality and the global stability of log-{S}obolev inequality on
  the sphere}, Adv. Math. \textbf{479} (2025), 110438.

\bibitem[CLT25b]{chen_2025b}
\bysame, \emph{Optimal stability of {H}ardy--{L}ittlewood--{S}obolev and
  {S}obolev inequalities of arbitrary orders with dimension-dependent
  constants}, 2025.

\bibitem[Chr14]{christ_2014}
M.~Christ, \emph{A sharpened {H}ausdorff--{Y}oung inequality}, 2014.

\bibitem[Chr17]{christ_2017}
\bysame, \emph{A sharpened {R}iesz--{S}obolev inequality}, 2017.

\bibitem[Chr21]{christ_2021}
\bysame, \emph{Young's inequality sharpened. {I}n: {G}eometric {A}spects of
  {H}armonic {A}nalysis ({P}. {C}iatti, {A}. {M}artini, eds.)}, Springer INdAM
  Ser. 45, pp.~261--298, Springer, Cham, 2021.




\bibitem[DEF{\etalchar{+}}24]{dolbeault_2024}
J.~Dolbeault, M.~J. Esteban, A.~Figalli, R.~L. Frank, and M.~Loss, \emph{A
  short review on improvements and stability for some interpolation
  inequalities}, 2024.

\bibitem[DEF{\etalchar{+}}25]{Dolbeault2023}
\bysame, \emph{Sharp stability for {S}obolev and log-{S}obolev inequalities,
  with optimal dimensional dependence}, Camb. J. Math. \textbf{13} (2025),
  no.~2, 359--430.

\bibitem[ENS22]{Engelstein2020}
M.~Engelstein, R.~Neumayer, and L.~Spolaor, \emph{Quantitative stability for
  minimizing {Y}amabe metrics}, Trans. Amer. Math. Soc. Ser. B \textbf{9}
  (2022), 395--414.

\bibitem[EG15]{Evans2015}
L.~C. Evans and R.~F. Gariepy, \emph{{Measure Theory and Fine Properties of
  Functions. Revised Edition}}, Chapman and Hall/CRC, New York, 2015.


\bibitem[FJ17]{figalli_2017}
A.~Figalli and D.~Jerison, \emph{Quantitative stability for the
  {B}runn-{M}inkowski inequality}, Adv. Math. \textbf{314} (2017), 1--47.

\bibitem[FMM18]{figalli_2018}
A.~Figalli, F.~Maggi, and C.~Mooney, \emph{The sharp quantitative {E}uclidean
  concentration inequality}, Camb. J. Math. \textbf{6} (2018), no.~1, 59--87.

\bibitem[FMP10]{figalli_2010}
A.~Figalli, F.~Maggi, and A.~Pratelli, \emph{A mass transportation approach to
  quantitative isoperimetric inequalities}, Invent. Math. \textbf{182} (2010),
  no.~1, 167--211.

\bibitem[FN19]{Figalli2015}
A.~Figalli and R.~Neumayer, \emph{Gradient stability for the {S}obolev
  inequality: the case $p\geq 2$}, J. Eur. Math. Soc. \textbf{21} (2019),
  no.~2, 319--354.

\bibitem[FZ22]{Figalli2022}
A.~Figalli and Y.~R.-Y. Zhang, \emph{Sharp gradient stability for the {S}obolev
  inequality}, Duke Math. J. \textbf{171} (2022), no.~12, 2407--2459.

\bibitem[Fra22]{Frank2023a}
R.~L. Frank, \emph{Degenerate stability of some {S}obolev inequalities}, Ann.
  Inst. H. Poincar{\'e} Anal. Non Lin{\'e}aire \textbf{39} (2022), no.~6,
  1459--1484.

\bibitem[Fra24]{Frank2023}
\bysame, \emph{The sharp {S}obolev inequality and its stability: {A}n
  introduction. {I}n: {G}eometric and {A}nalytic {A}spects of {F}unctional
  {V}ariational {P}rinciples. {C}etraro, {I}taly 2022 ({A}.~{C}ianchi,
  {V}.~{M}az'ya, and {T}.~{W}eth, eds.)}, pp.~1--64, Springer, Cham, 2024.

\bibitem[FL21]{frank_2021}
R.~L. Frank and E.~H. Lieb, \emph{Proof of spherical flocking based on
  quantitative rearrangement inequalities}, Ann. Sc. Norm. Super. Pisa Cl. Sci.
  (5) \textbf{22} (2021), no.~3, 1241--1263.

\bibitem[FNT25]{frank_2025}
R.~L. Frank, F.~Nicola, and P.~Tilli, \emph{The generalized {W}ehrl entropy
  bound in quantitative form}, J. Eur. Math. Soc. (2025), published online
  first.

\bibitem[FP24a]{Frank2024}
R.~L. Frank and J.~W. Peteranderl, \emph{Degenerate stability of the
  {C}affarelli--{K}ohn--{N}irenberg inequality along the {F}elli--{S}chneider
  curve}, Calc. Var. Partial Differential Equations \textbf{63} (2024), no.~44.

\bibitem[FP24b]{Frank2024B}
\bysame, \emph{The sharp $\sigma_2$-curvature inequality on the sphere in
  quantitative form}, 2024.



\bibitem[FMP08]{fusco_2008}
N.~Fusco, F.~Maggi, and A.~Pratelli, \emph{The sharp quantitative isoperimetric
  inequality}, Ann. of Math. (2) \textbf{168} (2008), no.~3, 941--980.


\bibitem[FZ17]{fusco_2017}
N.~Fusco and Y.~R.-Y. Zhang, \emph{A quantitative form of {F}aber-{K}rahn
  inequality}, Calc. Var. Partial Differential Equations \textbf{56} (2017),
  no.~5, 138.


\bibitem[GGRT24]{gomez_2024}
J.~G\'omez, A.~Guerra, J.~P.~G. Ramos, and P.~Tilli, \emph{Stability of the
  {F}aber-{K}rahn inequality for the short-time {F}ourier transform}, Invent.
  Math. \textbf{236} (2024), no.~2, 779--836.

\bibitem[GLZ25]{guerra2025}
A.~Guerra, X.~Lamy, and K.~Zemas, \emph{Sharp quantitative stability of the
  {M}\"obius group among sphere-valued maps in arbitrary dimension}, Trans.
  Amer. Math. Soc. \textbf{378} (2025), 1235--1259.

\bibitem[HWY08]{Hang2008}
F.~Hang, X.~Wang, and X.~Yan, \emph{Sharp integral inequalities for harmonic
  functions}, Commun. Pure Appl. Math. \textbf{61} (2008), no.~1, 54--95.

\bibitem[HWY09]{HANG2009}
\bysame, \emph{An integral equation in conformal geometry}, Ann. Inst. H.
  Poincar{\'e} C Anal Non Lin{\'e}aire \textbf{26} (2009), no.~1, 1--21.

\bibitem[HST22]{vanhintum_2022}
P.~van Hintum, H.~Spink, and M.~Tiba, \emph{Sharp stability of
  {B}runn-{M}inkowski for homothetic regions}, J. Eur. Math. Soc. \textbf{24}
  (2022), no.~12, 4207--4223.

\bibitem[HST24]{vanhintum_2024}
\bysame, \emph{Sharp quantitative stability of the planar {B}runn-{M}inkowski
  inequality}, J. Eur. Math. Soc. \textbf{26} (2024), no.~2, 695--730.


\bibitem[IZ25]{ignat2025}
L.~I. Ignat and E.~Zuazua, \emph{Optimal convergence rates for the finite
  element approximation of the {S}obolev constant}, 2025.

\bibitem[KP]{koenig2025}
T.~K\"onig and J.~W. Peteranderl, \emph{An almost almost {S}chur lemma,
  \textit{in preparation}}.

\bibitem[MRB22]{mazari_2022}
I.~Mazari and D.~Ruiz-Balet, \emph{Quantitative stability for eigenvalues of
  {S}chr\"odinger operator, quantitative bathtub principle, and application to
  the turnpike property for a bilinear optimal control problem}, SIAM J. Math.
  Anal. \textbf{54} (2022), no.~3, 3848--3883.

\bibitem[Neu19]{Neumayer2019}
R.~Neumayer, \emph{A note on strong-form stability for the {S}obolev
  inequality}, Calc. Var. Partial Differential Equations \textbf{59} (2019),
  no.~25.

\bibitem[NP25]{nobili2025}
F.~Nobili and D.~Parise, \emph{Quantitative stability of {S}obolev inequalities
  on compact {R}iemannian manifolds}, Int. Math. Res. Not. \textbf{2025}
  (2025), no.~1, rnae269.

\bibitem[Rat19]{Ratcliffebook}
J.~G. Ratcliffe, \emph{Foundations of {H}yperbolic {M}anifolds. {T}hird
  {E}dition}, Graduate Texts in Mathematics 149, Springer, Cham, 2019.

\bibitem[SW90]{Stein1990}
E.~M. Stein and G.~Weiss, \emph{{Introduction to Fourier Analysis on Euclidean
  Spaces. Sixth Printing}}, Princeton University Press, 1990.

\bibitem[WZ25]{wang2025}
G.~Wang and M.~Zhang, \emph{Stability of spinorial {S}obolev inequalities on
  $\mathbb{S}^n$}, 2025.

\end{thebibliography}

    \newcommand{\etalchar}[1]{$^{#1}$}
\providecommand{\bysame}{\leavevmode\hbox to3em{\hrulefill}\thinspace}

\end{document}